\documentclass[a4paper]{amsart}

\usepackage{amscd,amsthm,amsfonts,amssymb,amsmath}
\usepackage[colorlinks=true]{hyperref}
\usepackage{fullpage}
\usepackage[all]{xy}

    \newcommand{\BA}{{\mathbb {A}}} \newcommand{\BB}{{\mathbb {B}}}
    \newcommand{\BC}{{\mathbb {C}}} 
     \newcommand{\BF}{{\mathbb {F}}}

    \newcommand{\BQ}{{\mathbb {Q}}} \newcommand{\BR}{{\mathbb {R}}}

     \newcommand{\BZ}{{\mathbb {Z}}}

    \newcommand{\CO}{{\mathcal {O}}}

    \newcommand{\fc}{{\mathfrak{c}}}

     \newcommand{\fn}{{\mathfrak{n}}}
     \newcommand{\fp}{{\mathfrak{p}}}

     \newcommand{\fH}{{\mathfrak{H}}}

     \newcommand{\fN}{{\mathfrak{N}}}

     \newcommand{\fX}{{\mathfrak{X}}}

    \newcommand{\Aut}{{\mathrm{Aut}}}

    \newcommand{\cond}{\mathrm{cond^r}}

    \newcommand{\diag}{{\mathrm{diag}}}
     
    \newcommand{\End}{{\mathrm{End}}}

    \newcommand{\Gal}{{\mathrm{Gal}}} \newcommand{\GL}{{\mathrm{GL}}}
    
    \newcommand{\Hom}{{\mathrm{Hom}}}
    
    \renewcommand{\Im}{{\mathrm{Im}}}

    \newcommand{\ord}{{\mathrm{ord}}} \newcommand{\rank}{{\mathrm{rank}}}
     \newcommand{\Pic}{\mathrm{Pic}}
     
    \newcommand{\PSL}{{\mathrm{PSL}}}
    \renewcommand{\mod}{\ \mathrm{mod}\ }
    
    \newcommand{\ram}{{\mathrm{ram}}}
    
    \newcommand{\Res}{{\mathrm{Res}}}

    \newcommand{\supp}{{\mathrm{supp}}}

      \newcommand{\Supp}{{\mathrm{Supp}}}
    \newcommand{\tor}{{\mathrm{tor}}}

    \font\cyr=wncyr10

    \newcommand{\Sha}{\hbox{\cyr X}}
    \newcommand{\wh}{\widehat}

    \newcommand{\ov}{\overline}

    \newcommand{\lra}{\longrightarrow}
    \newcommand{\ra}{\rightarrow}

    \theoremstyle{plain}
    \newtheorem{thm}{Theorem}[section] \newtheorem{cor}[thm]{Corollary}
    \newtheorem{lem}[thm]{Lemma}  \newtheorem{prop}[thm]{Proposition}
    \newtheorem {conj}[thm]{Conjecture} \newtheorem{defn}[thm]{Definition}

\theoremstyle{remark} \newtheorem{remark}[thm]{Remark}
\theoremstyle{remark} 
\theoremstyle{remark} 

    \newcommand{\adeles}{ad\'{e}les~}

    \newcommand{\cO}{\mathcal O}
    \numberwithin{equation}{section}

\begin{document}

\title{Horizontal non-vanishing of\\ Heegner points and toric periods}

\author{Ashay A. Burungale and Ye Tian}

\address{Ashay A. Burungale:  Institute for Advanced Study,
Einstein Drive, Princeton NJ 08540 and \newline
California Institute of Technology, 1200 E California Blvd, Pasadena, CA 91125 \footnote{current address}} \email{ashayburungale@gmail.com}

\address{Ye Tian: MCM, HLM, Academy of Mathematics and Systems Science, Chinese Academy of Sciences, Beijing 100190, China 
 and \newline
School of Mathematical Sciences, University of Chinese Academy of Sciences, Beijing 10049, China} \email{ytian@math.ac.cn}

\maketitle
\begin{abstract}
Let $F$ be a totally real number field and $A$ a modular $\GL_2$-type abelian variety over $F$.
Let $K/F$ be a CM quadratic extension. Let $\chi$ be a class group character over $K$ such that 
the Rankin-Selberg convolution $L(s,A,\chi)$ is self-dual with root number $-1$.
%$\epsilon(A,\chi)=-1$.
%We consider the non-triviality of the central-derivative $L'(1,A,\chi)$.
We show that the number of class group characters $\chi$ with bounded ramification
such that $L'(1, A, \chi) \neq 0$ increases with the absolute value of the discriminant of $K$.

We also consider a rather general rank zero situation.
Let $\pi$ be a cuspidal cohomological automorphic representation over $\GL_{2}(\BA_{F})$. Let $\chi$ be a Hecke character over $K$ such that
the Rankin--Selberg convolution $L(s,\pi,\chi)$ is self-dual with root number $1$.
We show that the number of Hecke characters $\chi$ with fixed $\infty$-type and  bounded ramification
such that $L(1/2, \pi, \chi) \neq 0$ increases with the absolute value of the discriminant of $K$.
%We consider the non-vanishing of the family of central-critical Rankin-Selberg L-values $L(\frac{1}{2},f\otimes \lambda\chi)$,
%as $\chi$ varies over the class group characters of $K$.

The Gross--Zagier formula and the Waldspurger formula relate the question to horizontal non-vanishing of Heegner points and toric periods, respectively. 
For both situations, the strategy is 
%essentially the same.
%The approach is 
geometric relying on the Zariski density of
CM points on self-products of a quaternionic Shimura variety.
The recent result \cite{Ts, YZ, AGHP} on the Andr\'e--Oort conjecture
is accordingly fundamental to the approach.
\\
\end{abstract}

\tableofcontents

\section{Introduction}
For a self-dual family of L-functions with root number $-1$ (resp. $+1$),
the central derivatives (resp. central L-values) are generically believed to be non-vanishing.
These values can often be expressed as  periods of automorphic forms over algebraic cycles in a Shimura variety. 
An appropriate density of such cycles may thus be a key towards the non-vanishing. 

An instructive set up arises from  self-dual Rankin--Selberg convolutions of a
fixed cuspidal automorphic representation of $\GL_2$ over a totally real field $F$ and  Hecke characters 
%$\chi$ 
over  CM quadratic extensions of $F$ with  fixed infinity type and bounded conductor.   Strictly speaking,  the theta series 
%$\pi_\chi$ of 
associated to these characters  have prescribed local behavior at a finite set of places of $F$ including all the archimedean places.
The number of such Hecke characters with the corresponding 
%derivative (resp. central) 
central derivative of the L-functions (resp. central L-value)
non-vanishing is expected to grow with the absolute value of the
discriminant of the CM extension. This is the horizontal variation we consider in the article.

A special case relates to the following. For a fixed modular abelian variety over $F$, a fundamental question is to understand the structure of the Mordell--Weil group of the abelian variety over varying extensions of $F$. A naive expectation is that the the rank of the Mordell--Weil group typically increases with the degree of the extensions as long as the root number over the extensions equals $-1$.  
In view of the CM theory and modular parametrisation of the abelian variety, perhaps a natural question is to study the Galois module structure of the Mordell--Weil group over Hilbert class fields of varying  CM quadratic extensions of $F$. This amounts to consider the twists of the abelian variety by class group characters of the varying CM quadratic extensions.  

We begin with the general setup. 

Let $\ov{\BQ}$ be an algebraic closure of the rationals and $\iota_{\infty}: \overline{\BQ} \to \BC$ an embedding.

Let $F$ be a totally real number field, $\cO_F$ the ring of integers, $I$ the set of infinite places and $\BA$ the ring of ad\'eles.
Let $\fn$ and $\fc$ be ideals of $\cO_F$.

Let $\pi$ be a cuspidal automorphic
representation of $\GL_{2}(\BA)$ with conductor $\frak{n}$ and unitary central character $\omega$.
We suppose that 
$\pi$ is cohomological. 
In particular, $\pi_\infty$
is a discrete series for $\GL_2(F_\infty)$ with weight $(k_v)_{v|\infty}$ such that 
$$k_{v} \equiv k_{v'} \mod{2}$$ for $v,v' | \infty$. 
%In other words, $\pi$ is cohomological. 

%Let $S$ be a finite set of places of $F$ containing all archimedean places and the finite places dividing $\frak{n}$.
%Let $F_{S}$ denote the product of corresponding local completions.
%Let $\kappa=\sum_{v|\infty} \kappa_v v \in \Z[\Sg]$ be an infinite type.
%For each $v\in S$, let $K_{0, v}/F_{v}$ be an étale quadratic extension with $K_{0,v}\cong \BC$ for $v\in I$.
%Let $\eta_v$ the corresponding quadratic character on $F_v^\times$.
%Let $S_{0} \subset S$ be the subset consisting of $v$ with $K_{0,v}$ split and $S_{1}$ the complement.

%Let $K_{0, S}=\prod_{v\in S} K_{0, v}$ and
%$$ U_{0, S}=\prod_{v\in S_{0}}\cO_{K_{0, v}}^\times \times \prod_{v\in S_{1}} K_{0, v}^\times.$$
%Suppose we are given a character
%$\chi_{0, S}: U_{0, S}\rightarrow \BC^\times$
%such that
% \begin{itemize}
% \item[(S1)] $\chi_{0, S} \cdot \omega_S=1$ on $F_S^\times\cap U_{0, S}$.
%\item[(S2)] $\chi_{0, S}(u)=1$ for all units $u$ in $F$.
    %\end{itemize}

%Let $K/F$ be a CM quadratic extension such that for all $v\in S$, $K_v\cong K_{0, v}$.
%The existence follows from the finiteness of $S$. We fix these isomorphisms and identify $K_v$ with $K_{0, v}$.

Let $K/F$ be a CM quadratic extension. 
Let $D_{K}$ be the discriminant of the extension $K/\BQ$ and $K_\BA$ the ring of adeles. Let $c$ denote the complex conjugation on $\BC$ which induces the unique non-trivial element of $\Gal(K/F)$ via the embedding $\iota_{\infty}$.
Let $\Sigma$ be a CM type of $K$. We often identify $\Sigma$ with $I$.
Let $\kappa \in \BZ[\Sigma \cup \Sigma c]$ be an infinity type over $K$ such that it arises as an infinity type of an arithmetic Hecke character over $K$. In other words, $\kappa=\kappa_{I}+\kappa^{I}(1-c)$ with $\kappa_{I}, \kappa^{I}\in \BZ_{\geq 0}[\Sigma]$ such that $\kappa_{I}=l\cdot \Sigma$ for an integer $l \geq 0$.
Let $\Pic_{K/F}^{\fc}$ denote the relative ideal class group with conductor $\fc$. In other words, it is the ring class group of $K$ with conductor $\fc$ modulo the image of the class group of $F$ \eqref{rcg}. 

Let $\Theta_{\kappa,\fc}$ be the collection of CM quadratic extensions over $F$ such that there exists a Hecke character $\chi_0:  K^\times \backslash K_\BA^\times\rightarrow \BC^\times$ over $K$ satisfying the following.
\begin{itemize}
\item[(C1)] $\omega\cdot \chi_{0}|_{\BA^\times}=1$,
\item[(C2)]
%$\pi_{\chi_{0}}$ is discrete series with weight
$\chi_0$ is with infinity type $\kappa$ and $\cond(\chi_{0})=\fc\cO_K$.
\end{itemize}
%Here $\pi_{\chi_{0}}$ denotes the theta lifting to $\GL_2(\BA)$ corresponding to the Hecke character $\chi_0$.

For $K \in \Theta_{\kappa,\fc}$, let $\frak{X}_{K,\kappa,\fc}$ denote the group of Hecke characters
%$\chi: K^\times \backslash K_\BA^\times\rightarrow \BC^\times$ 
$\chi$ over $K$ satisfying the conditions (C1) and (C2).
%\begin{itemize}
%\item[(C1)] $\omega\cdot \chi|_{\BA^\times}=1$,
%\item[(C2)] $\chi|_{U_{0, S}}=\chi_{0, S}$ and
%\item[(C3)] $\chi$ is unramified outside $S$.
%\end{itemize}
%Note that the infinity type of $\chi \in \frak{X}_{K,S}$ is determined by (C2) and consequently constant for given $K \in \Theta_S$.
%Upon the identification of $\Sigma$ with $I$, the infinity type is in fact independent of $K \in \Theta_{S}$.
%Regarding the existence of Hecke characters satisfying (C1)-(C3), we have the following.
%\begin{lem}
%For all but finitely many $K \in \Theta_{S}$, we have $\frak{X}_{K,S} \neq \varnothing$.
%Let $K \in \Theta_{S}$ with $\frak{X}_{K,S} \neq \varnothing$. Then,
The set $\frak{X}_{K,\kappa,\fc}$ is a homogenous space of the relative ideal class group $\Pic_{K/F}^{\fc}$. 
In what follows, we regard $\kappa_{I}$, $\kappa^{I}$ as elements of $\BZ_{\geq 0}[I]$ and fix them independent of $K$. 
% given by
%$$\Pic_{K/F}^{S}=\BA_{K}^\times/K^\times \BA^{\times}\wh{\cO}^\times_K \wh{F}^\times \prod_{v \in S_{1}} K_{v}^{\times}.$$
%\end{lem}

 For $\chi \in \frak{X}_{K,\kappa,\fc}$, we consider the Rankin--Selberg convolution $L(s,\pi,\chi)$ 
 \footnote{In this article, we follow automorphic normalisation unless otherwise stated.}. 
In view of (C1), the convolution is self-dual.  Let $\epsilon(\pi,\chi)$ be the corresponding root number. 
As $K$ varies, the generic vanishing or non-vanishing of the central L-values is conjectured to be controlled by the root number.
%We have the following variant of Goldfeld's conjecture.
\begin{conj}
\label{NV}
Let $F$ be a totally real number field, $\BA$ the ring of adeles and $\pi$ a cuspidal cohomological automorphic
representation of $\GL_{2}(\BA)$.
% with conductor $\frak{n}$.
Let $\fc \subset \cO_F$ be an ideal and $I$ the set of infinite places of $F$.
Let $\Theta_{\kappa,\fc}$ be the set of CM quadratic extensions $K/F$ for $\kappa \in \BZ[I \cup Ic]$ 
with $c \in \Gal(K/F)$ the non-trivial element and $\fX_{K,\kappa,\fc}$ the set of Hecke characters over $K$ satisfying the above conditions (C1) and (C2).
%Let $S$ be a finite set of places of $F$ containing all archimedean places and the finite places dividing $\frak{n}$.
%Let $S_1\subset S$ be the set of places $v|\frak{n}\infty$ and $K_{0, v}$ non-split.
%For $U_{0,S}$ as above, let $\chi_{0,S}$ be a character of $U_{0,S}$.
Let $i\in \{0,1\}$ be such that
%$$\prod_{v\in S_1} \epsilon(1/2, \pi_v, \chi_{0, v})\chi_{0, v}\eta_v(-1)= (-1)^{i}.$$
$$
\epsilon(\pi,\chi_{0})=(-1)^{i}
$$
for some $\chi_{0} \in \fX_{K,\kappa,\fc}$.

For any $\epsilon > 0$ and $K \in \Theta_{\kappa,\fc}$, we have
$$ \#\left\{\chi\in \frak{X}_{K,\kappa,\fc} \ \Big|\ L^{(i)}(1/2, \pi,  \chi)\neq 0 \right\}\gg_{\epsilon, \pi} |D_{K}|^{\frac{1}{2} - \epsilon}.$$
For $\Theta \subset \Theta_{\kappa,\fc}$ infinite, in particular
$$\lim_{K\in \Theta} \#\left\{\chi\in \frak{X}_{K,\kappa,\fc} \ \Big|\ L^{(i)}(1/2, \pi,  \chi)\neq 0 \right\}=\infty.$$
\end{conj}
%The case when $i=0$ (resp. $i=1$) is usually referred as the rank zero (resp. rank one) case.
%In this article, we essentially consider the rank zero case.

\begin{remark}
%(1). In view of the root number condition on $(\pi_{S_{0}},\chi_{0,S_{0}})$, it follows that the root number of the Rankin-Selberg convolution for the pairs
%$(\pi,\chi)$ with $\chi \in \frak{X}_{K,S}$ is also $(-1)^{i}$.
%The above variant of Goldfeld's conjecture thus concerns non-vanishing along a self-dual family with fixed root number.
%
%(2).
(1). For a positive proportion of $\chi\in \frak{X}_{K,\kappa,\fc}$, we have  
$$
\epsilon(\pi,\chi)=\epsilon(\pi,\chi_{0}).
$$ 
For any $0 < \epsilon' < 1/2$, from the Brauer-Siegel bound
$$
|D_{K}|^{1/2+\epsilon'}\gg_{\epsilon',\fc}
|\frak{X}_{K,\kappa,\fc}| \gg_{\epsilon', \fc} |D_{K}|^{1/2-\epsilon'}.
$$
%(Lemma 2.1 and \cite{Br}).
It may be further conjectured that the non-vanishing holds for a positive proportion of the twists.
In general, even the growth in the number of non-vanishing twists with the discriminant seems non-trivial to show. 

%(3). The conjecture does not seem to be explicitly stated in the literature. Nevertheless,
%it follows from folklore conjectures on the non-vanishing in self-dual families.

(2). The above conjecture does not seem to be explicitly stated in the literature. However, it is a consequence of general conjectures on the non-vanishing in self-dual families.
%experience seems to suggest the non-vanishing to be harder in the case of root number $-1$ than the case of root number $1$.

\end{remark}

%It is clear that the set $\fX_K$, if not empty, is a homogenous space of the relative ideal class group
%    $\Pic_{K/F}^{S}:=K_{\BA}^\times/K^\times \BA^{\times}\wh{\CO}^\times_K \wh{F}^\times \prod_{v \in S,K_{0, v} \ \text{is a field}} K_{v}^{\times}$.

%    \begin{lem} The set $\fX_K$ is not empty for all but finitely many quadratic field extensions $K$ of $F$ with $K_S\cong K_{0, S}$.
%    \end{lem}

%\begin{proof}Note that for all but finitely many totally imaginary quadratic extensions $K$ over $F$ we have that $\CO_K^\times=\CO_F^\times$. Assume that $K$ is such one, we now show that $\fX_K$ is a homogenous space of $\Pic_{K/F}^S$.

%First consider the case  that $\chi_{0, S}$ is of finite order, namely, $k=0$ and $\kappa_v=0$ for all $v|\infty$. Let $\fm$ be the conductor of $\chi_{0, S}$ and let $\epsilon$ be a Hecke character of $K_\BA^\times$ of conductor $\fm$, trivial on $K_\infty^\times$, and extending $\omega$. Let $\chi_{0, S}'=\chi_{0, S} \cdot \epsilon|_{U_{0, S}}$. Then there is a character $\chi'$ of $\BA_K^\times/K^\times K_\infty^\times \BA^\times U_\fm$ extending $\chi_{0, S}'$. The character $\chi'\cdot \epsilon^{-1}$ is the desired one. Moreover, for any two characters $\chi_i\in \fX_K$, it is clearly that $\chi_1/\chi_2$ is a character on $\Pic_{K/F}^{S}$. For general $\infty$-type, fix any Hecke character $\chi_0$ of such type and consider $\chi/\chi_0$. It reduces to the finite order case.

%\end{proof}
%We now turn to the rank one case.

We now describe the results.

In the case of root number $-1$, we restrict to $\pi$ with parallel weight two and consider finite order Hecke characters.
Suppose that $\pi$ is associated to a $\GL_2$-type abelian variety $A$ over the totally real field $F$ and that $\chi_{0}$
is of finite order.
% (thus is trivial over $K_v^\times$ for all $v|\infty$).
%Let $i=0, 1$ such that
%$\displaystyle{\prod_{v\in S, K_{0, v} \ \text{is a field}} \epsilon(A, \chi_{0, v})=(-1)^i.}$

Our main result is the following.

\begin{thm} \label{rank1}
Let $F$ be a totally real number field, $\BA$ the ring of adeles and $\pi$ a cuspidal automorphic
representation of $\GL_{2}(\BA)$ with parallel weight two.
Suppose that $\pi$ corresponds to a $\GL_2$-type abelian variety over $F$.
For an integral ideal $\fc$, let $\Theta_{\fc}$ be  an infinite set of CM quadratic extensions $K/F$ such that there exists a finite order Hecke character
$\chi_0$ over $K$ satisfying the above conditions (C1) and (C2) along with
$$
\epsilon(\pi,\chi_0)=-1.
$$
(1). For $\Theta \subset \Theta_\fc$ infinite, we have
$$\lim_{K\in \Theta} \#\Big\{ \chi\in \fX_{K,\fc} \Big|\ L'(1/2, \pi, \chi)\neq 0 \Big\}=\infty.$$
%\begin{cor}
%Let the notation and assumptions be as in Theorem \ref{rank1}.
(2). Let $\CO=\End_F(A)$ and suppose that $(\cO \otimes \BQ) \cap K = \BQ$ for all but finitely many $K \in \Theta$. 

Then, we have
$$\lim_{K\in \Theta} \#\Big\{ \chi\in \fX_{K,\fc} \Big| \ \dim A(\ov{F})^\chi=[\CO:\BZ], 
\Sha(A_{/H_{K, \fc}})^\chi\ \text{is finite} \Big\}=\infty.$$ 
Here $A(\ov{F})^\chi$ (resp. $\Sha(A/H_{K, \fc})^\chi$) denotes the $\chi$-isotypic component of $A(\ov{F})$ (resp. $\Sha(A_{/H_{K, \fc}})$) and $H_{K,\fc}$ the ring class field of $K$ with conductor $\fc$. 
%\end{cor}
\end{thm}
In view of the results of Kolyvagin (\cite{Ko}) and Nekov\'a\v{r} ™(\cite{N}), the second part follows from the first.

We describe a special case of Theorem \ref{rank1}.
\begin{cor}Let $E$ be an elliptic curve over $\BQ$ and $c$ a positive integer.
Let $\Theta$ be an infinite set of imaginary quadratic fields $K$ such that the Rankin--Selberg convolution $L(s, E, \chi)$ has root number 
$-1$
%$(-1)^i$
for at least one ring class character $\chi_{0}$ over $K$ 
with conductor $c$. For each $K\in \Theta$,  let $\fX_K$ denote the set of ring class characters over $K$ with conductor $c$ such that the Rankin--Selberg convolution $L(s, E, \chi)$
\footnote{Here we follow arithmetic normalisation.} has root number $-1$. 

Then, 
$$\lim_{K\in \Theta} \#\big\{ \chi\in \fX_K |\ L'(1, E, \chi)\neq 0 \big\}=\infty$$
and
$$ \lim_{K\in \Theta} \#\Big\{ \chi\in \fX_K \ \Big|\ \dim E(H_{K, c})^\chi=1,\
\Sha(E_{/H_{K, c}})^\chi\ \text{is finite}\ \Big\}=\infty.$$ 
Here $H_{K,c}$ denotes the ring class field of $K$ with conductor $c$ and $E(H_{K,c})^\chi$ (resp. $\Sha(E/H_{K, \fc})^\chi$) denotes the $\chi$-isotypic component of $E(H_{K,c})$ 
(resp. $\Sha(E_{/H_{K, c}})$). 
\end{cor}
\begin{remark}
A related question is whether  $$\displaystyle{\lim_{K\in \Theta} \rank_\BZ \left( \frac{E(H_{K, c})}
{\sum_{c'|c, c'\neq c} E(H_{K, c'})}\right)=\infty}.$$
The approach in the article does not seem to suffice. 
\end{remark}

In the case of classical Heegner hypothesis, we have the following consequence.
\begin{cor} Let $E$ be an elliptic curve over $\BQ$. Let $\Theta$ be  an infinite set of imaginary quadratic fields $K$ such that the base change $L(s, E_K)$ has root number $-1$. 

Then,  
$$\lim_{K\in \Theta} \rank_\BZ E(H_K) =\infty . 
$$
\end{cor}

We now turn to the case of root number $+1$.
In this case, our result is rather general with mild hypothesis on the weight of the pair $(\pi,\chi_0)$ satisfying the conditions (C1) and (C2).
\begin{thm}\label{rank0}
Let $F$ be a totally real number field, $\BA$ the ring of adeles and $\pi$ a cuspidal 
cohomological automorphic
representation of $\GL_{2}(\BA)$.
%Suppose that $\pi_\infty$
%is a discrete series for $\GL_2(F_\infty)$ of weight $(k_v)_{v|\infty}$ with the same parity.
%Suppose that $\pi$ corresponds to a $\GL_2$-type abelian variety over $F$.
For an integral ideal $\fc$, let $\Theta_{\kappa,\fc}$ be  an infinite set of CM quadratic extensions $K/F$ such that there exists a Hecke character
$\chi_{0} \in \fX_{K,\kappa,\fc}$ over $K$ satisfying the conditions (C1) and (C2) along with
$$
\epsilon(\pi,\chi_0)=+1.
$$
If $k_{\sigma} > l+2\kappa_{\sigma}$ for all $\sigma \in I$, then suppose that $\pi$ is with parallel weight two.

For $\Theta \subset \Theta_{\kappa,\fc}$ infinite, we have
$$\lim_{K\in \Theta} \#\Big\{ \chi\in \fX_{K,\kappa,\fc} \Big|\ L(1/2, \pi, \chi)\neq 0 \Big\}=\infty.$$
\end{thm} 
When $F=\BQ$ and $\fc=1$, the non-vanishing in Theorem \ref{rank1} goes back to Templier. 
For the results and various approaches to the non-vanishing, we refer to Templier's articles \cite{Te1}, \cite{Te} along with references therein.
Under the existence of enough small split primes in the imaginary quadratic fields, Michel--Venkatesh (\cite[Thm. 3]{MV1}) had earlier obtained analogous results. 
When the Rankin--Selberg convolution corresponding to pair $(\pi,\chi)$ is self-dual with root number 
$+1$ and the underlying quaternion algebra is indefinte, the non-vanishing in Theorem \ref{rank0} covers all the possibilities for the weight of the pair.  
When the underlying quaternion algebra is totally definite, the result covers key case of arithmetic interest with $\pi$ being of parallel weight $2$. In this case, the result is quantitative and goes back to Michel--Venkatesh  (\cite{MV1}, \cite{MV}) for $F=\BQ$ and $\fc=1$. 
%arithmetic cases.  
When all components of the weight of Hecke characters are at least the weight of corresponding component of $\pi$ i.e. $$l+2\kappa_{\sigma} \geq k_{\sigma}$$ for all $\sigma \in \Sigma$, the non-vanishing was recently considered by Burungale--Hida (\cite{BH}).
%In view of the Waldspurger formula, the case leads to analysis to toric periods on a Hilbert modular variety. For the general case, we are lead to consider periods on a quaternionic Shimura variety.
%\begin{thm}
%\label{definite}
%Let $E$ be an elliptic curve over $\BQ$ and $c$ a non-zero integer.   Let $\Theta$ be an infinite set of imaginary
%quadratic fields $K$ such that $L(s, E, \chi)$ has sign $+1$ for at least one ring class character $\chi$  of exact
%conductor $c$ over $K$. For each $K\in \Theta$,  denote by $\fX_K$  the set of ring class characters of $K$ of exactly order $c$ such that
%$L(s, E, \chi)$ has sign $+1$. Then
%$$\lim_{K\in \Theta} \#\Big\{ \chi\in \fX_K \Big|\ L(1, E, \chi)\neq 0 \Big\}=\infty$$
%and
%$$ \lim_{K\in \Theta} \#\Big\{ \chi\in \fX_K \ \Big|\ \dim E(H_{K, c})^\chi=0,\
%\Sha(E/H_{K, c})^\chi\ \text{is finite}\ \Big\}=\infty.$$
%\end{thm}

We now describe the strategy.
It is automorphic/ geometric. The approach is based on the Gross--Zagier formula and the Waldspurger formula. 
An essential role is played by Zariski density of well chosen CM points on
self-products of a quaternionic Shimura variety. The latter crucially relies on the recent proof of the Andr\'e--Oort conjecture for
abelian type Shimura varieties due to Tsimerman (\cite{Ts}) building on earlier work towards the conjecture along with the recent results of Andreatta--Goren--Howard--Pera (\cite{AGHP}) and Yuan--Zhang (\cite{YZ}) towards the averaged Colmez conjecture. 

We follow the formulation of  the Gross--Zagier formula (\cite{GZ}) and the Waldspurger formula (\cite{Wa}) due to Yuan--Zhang--Zhang (\cite{YZZ}). In fact, their formulation plays an underlying role throughout. 
To begin with, the very formulation allows us an access to a general self-dual setup. 
%The approach bears striking similarities and differences in the rank $0$ and rank $1$ case.

%Recall that the $L$-function corresponding to the pair
%$(E,\chi)$ is self-dual. Let $\sgn(E,\chi)$ denote the root number.
Let
$$
\fX_{K}^{\mp}=\big{\{} \chi \in \fX_{K,\kappa,\fc}\big{|} \epsilon(\pi,\chi)=\mp 1\big{\}}.
$$
The approach crucially relies on the root number. Let $S$ be the set of places of $F$ dividing $\fn\fc\infty$. For a Hecke character $\chi=\otimes \chi_v$ over $K$ 
%of $\Cl_K$
viewed as a character of $K_\BA^\times$ via class field theory, let 
$$\chi_S=\otimes_{v\in S}\chi_v: \prod_{v\in S} K_v^\times \ra \BC^\times$$ be its $S$-component. 
Some of the notation used here and in the sketch below is not followed in the rest of the article.

The set $\fX_{K}^{-}$ admits a finite partition according to the $S$-component  of the Hecke characters. 
For a fixed $S$-component $\chi_{0, S}$, there exists an incoherent quaternion algebra $\BB/\BA$ and a so-called $\chi_{0, S}$-toric line $V$ in $\pi_\BB$ such that $L'(1/2, \pi, \chi)\neq 0$ if and only if the Heegner point
$$P_f(\chi):=\int_{K_\BA^\times/K^\times \BA^\times} f(t) \chi(t) dt$$
is non-torsion for any non-zero vector $f\in V$. 
This is an inexplicit version of the Gross--Zagier formula (\cite{YZZ}).
Here $\pi_\BB$ denotes  the  Jacquet--Langlands transfer of $\pi$ to $\BB^{\times}$. 
Moreover, we fix an  embedding of $K_\BA$ into $\BB$ 
%for any character $\chi\in \fX_K^-$ with a fixed 
dependent on the $S$-component $\chi_{0, S}$. 
In the above expression for the Heegner point, $f$ is viewed as a modular parametrisation $f: X_{U} \ra A$ for Shimura curve $X_U$ arising from $\BB$ with level $U$ specified below and an abelian variety $A$ corresponding to $\pi$.
 % and by which view $K^\times$ as a $\BQ$-subalgebra of $B$.

The space $V$ contains non-zero vectors and we fix such a form $f$ from now. 
Let $U$ be an open compact subgroup of $\BB^{(\infty),\times}$ such that $f$ is $U$-invariant. 
Here $\BB^{(\infty)}$ denotes the finite part. 
Let $X_U$ be the  Shimura curve of level $U$. We can choose $U$ such that for any pair $(K, \chi)$ with the fixed type $\chi_{0, S}$, the chosen embedding $K_\BA$ into $\BB$ induces a map $\varphi_K: \Pic_{K/F}^{\fc}\ra X_U$. The image of the map is referred as the CM points arising from
the relative class group $\Pic_{K/F}^{\fc}$.

%Let $\fp_0|p$ be the prime of $\BZ[\chi_{0, S}]$ induced by $\iota_p$.
An essential point is to study the non-vanishing of Heegner points $P_{f}(\chi)$ as the pair
$(K,\chi)$-varies with the $S$-component being $\chi_{0,S}$.
% (thus $\fp_\chi|\fp_0$).
The non-vanishing of $P_f(\chi)$ begins with Fourier analysis on $\Pic_{K/F}^{\fc}$ 
and it's relation with the CM points arising from $\Pic_{K/F}^{\fc}$. 
%induced map $\varphi_{K}$.
Here we consider $f$ as a function on $\Pic_{K/F}^{\fc}$ with values in the Mordell--Weil group $A(H_{K,\fc})$ for the ring class field $H_{K,\fc}$.
% with conductor $\fc$.
%Here $\varphi_{K}'$ is a lift of $\varphi_K$.
%Recall that $X_U$ is a finite set whereas the class number grows with the discriminant of $K$.
Based on Shimura's reciprocity law and control on annihilation of a $\Gal(H_{K,\fc}/K)$-representation, we reduce the non-vanishing to
an Ax--Lindemann type functional independence for functions induced by
the modular form on
the class group $\Pic_{K/F}^{\fc}$.
%Based on the geometric interpretation of nearly holomorphic modular forms,
% (\cite[\S2.3]{U}),
The independence turns out to be closely related to the Zariski density of well-chosen CM points on a self-product of the Shimura curve $X_U$.
Here it is crucial to consider arbitrary self-products of the Shimura curve.
%We formulate a conjecture regarding such a Zariski density (Conjecture A).
%The conjecture is later shown to follow from the recent progress on the André-Oort conjecture.
The CM points in consideration are the images of the CM points arising from the ideal classes in the CM quadratic extensions and we consider the density as the CM quadratic extensions vary in the infinite subset $\Theta$.
%It seems delicate to directly study the density.
%As indicated before,
We show that the Zariski density follows from the Andr\'{e}--Oort conjecture for the self-product of the Shimura curve  $X_U$.

As the Fourier inversion works well only over the $\BQ$-rational space 
$A(H_{K,\fc})_{\BQ}:=A(H_{K,\fc}) \otimes_{\BZ} \BQ$, we need to keep track of the torsion $A(H_{K,\fc})_\tor$. Based on the results for Galois image of a $\GL_2$-type abelian variety, we show that the torsion 
$$
\bigcup_{K} A(H_{K,\fc})_\tor
$$
is finite as $K$ varies over the CM quadratic extensions of $F$. The finiteness allows the approach to go through.

  The set $\fX_{K}^{+}$ admits a finite partition according to the $S$-component  of the Hecke characters.
   %in $\fX_K^+$. 
   For a fixed $S$-component $\chi_{0, S}$, there exists a quaternion algebra $B/F$ and a so-called $\chi_{0, S}$-toric line $V$ in $\pi_B$
   %,  the  Jacquet-Langlands transfer of $\pi$ to $B_{\BA}^{\times}$, 
   such that $L(1/2, \pi, \chi)\neq 0$ if and only if the toric period 
$$P_f(\chi):=\int_{K_\BA^\times/K^\times \BA^\times} f(t) \chi(t) dt$$
is non-zero for any non-zero vector $f\in V$. 
This is an inexplicit version of the Waldspurger formula (\cite{YZZ}). 
Here $\pi_B$ denotes  the  Jacquet--Langlands transfer of $\pi$ to $B_{\BA}^{\times}$ for $B_{\BA}=B \otimes_{F} \BA$. 
Moreover, we fix an  embedding of $K_\BA$ into $B_{\BA}^\times$ 
%for any character $\chi\in \fX_K^-$ with a fixed 
dependent on the $S$-component $\chi_{0, S}$. 
In the rest of the sketch, we suppose that $B$ is not totally definite. 
In the totally definite case, we closely follow the approach in \cite{MV1}. 
%Here for any character $\chi\in \fX_K^+$ with a fixed type $\chi_{0, S}$ we have fixed a nice embedding of $K$ into $B$ and by which view $K^\times$ as a $F$-subalgebra of $B$. 

The space $V$ contains non-zero vectors and we fix such a form $f$ from now. 
Let $U$ be an open compact subgroup of $B_{\BA}^{(\infty),\times}$ such that $f$ is $U$-invariant. 
Here $B_{\BA}^{(\infty)}$ denotes the finite part. 
Let $X_U$ be the quaternionic Shimura variety of level $U$. We can choose $U$ such that for any pair $(K, \chi)$ with the fixed type $\chi_{0, S}$, the chosen embedding $K_\BA$ into $B_\BA$ induces a map $\varphi_K: \Pic_{K/F}^{\fc}\ra X_U$. The image of the map is referred as the CM points arising from
the relative class group $\Pic_{K/F}^{\fc}$.

%Let $U$ be a open compact subgroup of $\wh{B}^\times$ such that forms in $V$ are $U$-invariant and let $X_U$ be the quaternionic Shimura variety of level $U$. We can choose $U$ such that for any pair $(K, \chi)$ with the fixed type $\chi_{0, S}$, the chosen embedding $K$ into $B$ induces a map $\varphi_K: \Pic_{K/F}^{\fc}\ra X_U$. The image of the map is referred as the CM points arising from $\Pic_{K/F}^{\fc}$.

An essential point is to study the non-vanishing of toric periods $P_{f}(\chi)$ as the pair
$(K,\chi)$-varies with the $S$-component being $\chi_{0,S}$.
The non-vanishing of $P_\chi(f)$ begins with Fourier analysis on $\Pic_{K/F}^{\fc}$ and its relation with the CM points arising from the class group $\Pic_{K/F}^{\fc}$.
%Recall that $X_U$ is a finite set whereas the class number grows with the discriminant of $K$.
%The form $f$ is a nearly holomorphic form on $X_U$.
Based on Shimura's reciprocity law, we reduce the non-vanishing to
to an Ax--Lindemann type functional independence for functions induced by
the modular form on
the class group $\Pic_{K/F}^{\fc}$.
%Based on the geometric interpretation of nearly holomorphic modular forms,
% (\cite[\S2.3]{U}),
The independence is essentially equivalent to the Zariski density of well-chosen CM points on a self-product of the quaternionic Shimura variety $X_U$.
%Here we need to consider arbitrary self-products of the quaternionic Shimura variety.
%We formulate a conjecture regarding such a Zariski density (Conjecture A).
%The conjecture is later shown to follow from the recent progress on the André-Oort conjecture.
The CM points in consideration are the images of CM points arising from the ideal classes in the CM quadratic extensions
by a skewed diagonal map and we consider the density as the CM quadratic extensions vary in the infinite subset $\Theta$.
%It seems delicate to directly study the density.
We show that it follows from the Andr\'{e}--Oort conjecture for the self-product of the quaternionic Shimura variety $X_U$.

As evident from the sketch, there is an intriguing analogy among the case of root number $-1$ and the case of root number $+1$. The analogy seems to resonate resemblance among the Gross--Zagier formula and the Waldspurger formula. 
On the other hand, there are striking differences. Here we only mention that the Zariski density of CM points used  in the case of root number $-1$ differs from the one in the case of root number $+1$ (Theorem \ref{density1} and Theorem \ref{density2}, respectively).
 
%{\bf{Analytic difference rank 1, rank 0}}

As indicated earlier, a special case of Theorem \ref{rank0} was proven in \cite{BH}. In fact, our motivation partly came from \cite{BH}. The initial attempt was to seek analogue of the non-vanishing in the special case to the case of root number $-1$. Our approach is based on the root number $+1$ case in \cite{BH}. 

The non-vanishing over the family of twists by class group characters  has been considered in the literature in the case of root number $-1$. 
Here we mention \cite{MV1}, \cite{Te1}, \cite{Te} and refer to these articles for a quantitative version of Theorem \ref{rank1} in the case $F=\BQ$ and $\fc=1$ under certain hypotheses. The results in \cite{Te} are perhaps closest to our study of the Heegner points. 
%In \cite{MY}, a quantitative analog of the theorem is considered for certain Hecke characters. 
In these articles, the approach perhaps has more analytic/ ergodic flavour. The key ingredients include 
\begin{itemize}
\item[--] an explicit version of the Gross--Zagier formula, 
\item[--] a subconvex bound for central derivatives of the Rankin--Selberg L-functions, 
\item[--] an equidistribution of CM points on a Shimura curve with respect to the complex-analytic topology and 
\item[--] a lower bound for N\'eron--Tate height of non-torsion Heegner points. 
\end{itemize}
In the split case, the geometric argument only uses Burgess bound for character sums which controls the average of toric period against Eisenstein series.  
It seems suggestive to compare our approach with the one in \cite{Te}. 
Here we only mention the following. For the analytic approach, the level of complexity seems to increase significantly while moving from the case of root number $+1$ to the case of root number $-1$. Somewhat surprisingly, 
the level of complexity seems to change only slightly in our approach at the expense of a non-quantitative result.

As is clear from the sketch, the perspective on the non-vanishing based on the Andr\'e--Oort conjecture is likely to admit several generalisations. 
We hope to investigate them in the near future. 
Arithmetic Gan--Gross--Prasad conjecture often relates 
non-triviality of a central derivative of a self-dual Rankin--Selberg L-function to the non-triviality of a period of a cycle-valued automorphic form. 
%As mentioned before, the Andr\'e-Oort conjecture is 
%also now known for Abelian type Shimura varieties. 
It would be especially interesting to consider the case of higher rank unitary groups. 
Along a different direction, a natural question would to explore horizontal non-vanishing modulo $p$ for a fixed prime $p$. 
In the setup of the current article, we refer to \cite{BHT} for certain results. 

As opposed to the horizontal case, the vertical or Iwasawa-theoretic non-vanishing of Heegner points has been much studied. Even though contrasting in appearance, there seem to be mysterious analogies among the current approach and the approach based on Chai--Oort rigidity principle (\cite{Bu2}, \cite{Bu3} and \cite{BD}). We refer to \cite{Bu1} for an overview of the latter. 

The article is organised as follows. In \S2, we describe results on the Zariski density of a class of CM points on self-products of a quaternionic Shimura variety. In \S2.1, we give an explicit description of special subvarieties of self-products of a quaternionic Shimura variety. In \S2.2, we prove the results on Zariski density.
In \S3, we consider the root number $-1$ case and describe the non-vanishing of Heegner points.
In \S3.1, we introduce the setup and fix an incoherent algebra $\BB$.
In \S3.2, we prove the horizontal non-vanishing of Heegner points arising from $\BB$.
In \S4, we consider the root number $+1$ case and describe the non-vanishing of toric periods.
In \S4.1, we prove the horizontal non-vanishing as long as the underlying quaternion algebra is not totally definite.
In \S4.2, we prove the horizontal non-vanishing when the underlying quaternion algebra is totally definite.
In \S5, we prove the main results based on \S2--\S4 along with the existence of suitable torus embeddings into the underlying quaternion algebras.

In the exposition, we often suppose familiarity with the formalism in \cite{YZZ}. For an overview, we refer to \cite[Ch. 1]{YZZ}.

\section*{Acknowledgement}
We are grateful to Haruzo Hida for encouragement and suggestions. 
The topic was initiated by the joint work of the first-named author with Haruzo Hida regarding horizontal non-vanishing (\cite{BH}). We thank Li Cai for his assistance and Hae-Sang Sun for stimulating conversations regarding horizontal non-vanishing. 
We thank Nicolas Templier for helpful comments and suggestions. 
We also thank Farrell Brumley, Henri Darmon, Najmuddin Fakhruddin, Dimitar Jetchev, Mahesh Kakde, Chandrashekhar Khare, Philippe Michel, C.-S. Rajan, Jacques Tilouine, Vinayak Vatsal, Xinyi Yuan, Shou-Wu Zhang and Wei Zhang 
for instructive conversations about the topic. 

The first named author is grateful to MCM Beijing for the continual warm hospitality. The article was conceived in Beijing during his first MCM visit \footnote{the summer of 2015}.

Finally, we are indebted to the referee.

%\begin{Acknowledgement} The 
%\end{Acknoledgemet}
%\begin{thk} We
%\end{thk}

\section*{Notation}
We use the following notation unless otherwise stated.

For a finite abelian group $G$, let $\widehat{G}$ denote $\overline{\BQ}^\times$-valued character group of $G$. For a $\BZ$-module $A$, let $\widehat{A}=A \otimes_{\BZ} \widehat{\BZ}$ for $\widehat{\BZ}=\varprojlim_{n} \BZ/n$.

For a number field $L$, let $\cO_L$ be the corresponding integer ring and $D_L$ the discriminant.
Let $L_+$ denote the totally positive elements in $L$.
Let $G_{L}=\Gal(\overline{\BQ}/L)$ denote the absolute Galois group over $L$. Let $\BA_L$ denote the adeles over $L$. For a finite subset $S$ of places in $L$, let $\BA_{L}^{(S)}$ denote the adeles outside $S$ and $\BA_{L,S}$ the $S$-part.
When $L$ equals the rationals or an underlying totally real field, we drop the subscript $L$.
For a $\BQ$-algebra $C$, let $C_{\BA}=C\otimes_{\BQ} \BA$. Let $\widehat{C}^{(S)}$ (resp.
$C_{S}$) denote the part outside $S$ (resp. $S$-part) of
$C_{\BA}$.

For a place $v$ of $F$ for a totally real field $F$ and a quadratic extension $K_{v}/F_{v}$, let $\eta_{v}$ denote the corresponding quadratic character. For a quaternion algebra $B_{v}/F_{v}$, let $\epsilon(B_{v})$
denote the corresponding local invariant.
For a CM quadratic extension $K/F$ and an integral ideal $\fc$ of $F$, let $H_{K,\fc}$ be the ring class field with conductor $\fc$ and
$\Pic_{K/F}^{\fc}$ the relative ring class group with conductor $\fc$.
Let $h_{K}$ (resp. $h_{K,\fc}$) be the relative ideal class number of $K$ (resp. $H_{K,\fc}$) over $F$.

\section{Zariski density of CM points}

%\s{\bf Change later}: Add more details.

In this section, we consider the Zariski density of well-chosen CM points on self-product of a quaternionic Shimura variety.
The density plays an underlying role in the non-vanishing.

\subsection{Special subvarieties}
In this subsection, we give an explicit description of a class of special subvarieties of self-product of a quaternionic Shimura variety.

Let the notation and assumptions be as in \S1. In particular, $F$ is a totally real field.
Let $B$ be a quaternion algebra over $F$ and suppose that it is not totally definite. 
Let $G$ be the reductive group $\Res_{F/\BQ}B^{\times}$.

Let $X=\varprojlim X_{U}$ be the quaternionic Shimura variety associated to Shimura datum corresponding to $G$.
Here $U \subset G(\BA^{(\infty)})$ is an open compact subgroup and corresponds to level of the Shimura variety. 
Recall that $X$ is defined over reflex field corresponding to the Shimura datum. 
Here we only recall that in the case of Shimura curves, the reflex field is nothing but the totally real field $F$.

In what follows, we consider self-product of the Shimura variety $X_U$ for a
sufficiently small level $U$. Let $r=\dim(X_{U})$. Recall that $r$ equals the number of archimedean places of the totally real field $F$ which split the quaternion algebra $B$. The complex points of these varieties are given by
$$
X_{U}(\BC) =  G(\BQ) \backslash (\BC - \BR)^{r} \times G(\BA^{(\infty)}) / U .
% Sh(\C)= G(\Q) \backslash X \times G(\A^f) / \overline{Z(\Q)}.\\
$$
%Here $X \simeq (\BC - \BR)^{r}$. 
In particular, these varieties admit complex uniformisation by $\mathfrak{H}^{r}$ for the upper half plane $\mathfrak{H}$.

We first consider the case of special subvarieties of $X_U$ itself (\cite[Prop. 2.2]{BH}).
\begin{lem}
\label{special1}
The special subvarieties of the quaternionic Shimura variety $X_U$ arise from the reductive group $\Res_{E/\BQ}D^{\times}$ for
a subfield $E$ of the totally real $F$ and a quaternion algebra $D_{/E}$ with an embedding into the quaternion $B_{/F}$.

\end{lem}

We now consider the case of self-product.
Let $S$ be a finite set. For $s \in S$, let $\pi_s$ be the projection to the $s$-component of the self-product $X_{U}^S$.
\begin{lem}
\label{special2}
Let $Z \subset X_{U}^{S}$ be a special subvariety with dominant projections $\pi_s$ onto an
irreducible component of $X_{U}$ for all $s \in S$.

Then, S has a partition $(S_{1},...,S_{k})$ such that $Z$ is a product of subvarieties $Z_i$ of $X_{i}=X_{U}^{S_{i}}$ which are the image of
$$  \fH^{r} \rightarrow (\fH^{r} /U)^{S_i}     $$ under the map
$$   \tau \mapsto ([g_{\sigma}(\tau)])_{\sigma \in S_{i}}.     $$
Here $\fH$ denotes the upper half complex plane and $g_{\sigma} \in G(\BQ)$ for $\sigma \in S_{i}$.
\end{lem}
\begin{proof}
The classification is a consequence of Lemma \ref{special1} and Goursat's lemma. 
For details, we refer to \cite[Prop. 2.3]{BH}.
\end{proof}

\subsection{Zariski density} In this subsection, we consider the Zariski density of well-chosen CM points on self-product of a quaternionic Shimura variety. For the relevant CM theory, we refer to \cite{Sh}.

Let the notation and assumptions be as in \S2.1. In particular, $X_U$ denotes a quaternionic Shimura variety with level $U$. 

Let $K/F$ be a CM quadratic extension such that there exists an embedding $\iota_K : K \hookrightarrow B$. Let $\fc$ be an ideal of $\cO_F$. In view of the CM theory, the embedding gives rise to a map 
$$\varphi_{K,\fc}: \Pic_{K/F}^{\fc} \ra X_{U/\ov{\BQ}}.$$ The image is usually referred as the CM points with CM by $K$ and conductor $\fc$. 
For $\sigma \in \Pic_{K/F}^{\fc}$, let $P_{K,\fc}^{\sigma}$ be the corresponding CM point arising from $\varphi_{K,\fc}$. 
Let $P_{K,\fc}$ be the image of the identity ideal class. 

Our first result regarding the Zariski density is the following. 

\begin{thm}\label{density1} 
Let $F$ be a totally real field and $\fc$ an ideal of the integer ring $\cO_F$. 
Let $B$ be a quaternion algebra over $F$ such that it is not totally definite and $X_U$ the corresponding quaternionic Shimura variety with level $U$. 
Let $\Theta$ be an infinite set of CM quadratic extensions $K/F$ as above. Let
$\ell$ be a positive integer and
$$
\Delta_{\ell,\fc}=\{ (P_{K,\fc}^{\sigma_{i}})_{1 \leq i \leq \ell}| \sigma_{i} \in \Pic_{K/F}^{\fc}, K \in \Theta \}
$$
a subset of CM points in the
$\ell$-fold self-product $X_{U/\overline{\mathbb{Q}}}^{\ell}$. 

Then, $\Delta_{\ell,\fc}$ is Zariski dense.

\end{thm}

\begin{proof}

For simplicity of notation, we restrict to the case $\fc=1$ and let $\Delta_{\ell}=\Delta_{\ell,1}$. The details are identical for the general case. 

Let $W\subset X_{U}^{\ell}$ denote the Zariski closure of the CM points $\Delta_{\ell}$.
From CM theory, each irreducible component of $X$ contains infinitely many CM points.
Let $I$ be an irreducible component of $W$ containing an infinite subset $T_\ell \subset \Delta_\ell$ of CM points.

The Andr\'{e}--Oort conjecture implies that
$I$ is a special subvariety of the self-product $X_{U}^{\ell}$.
As $X_{U/\overline{\mathbb{Q}}}^{\ell}$ is of abelian type, we recall that the Andr\'{e}--Oort conjecture has been proven unconditionally (\cite{AGHP}, \cite{Ts} and \cite{YZ}).
In view of Lemma \ref{special2}, we have an explicit list of the possibilities for $I$.
\vskip2mm
{\bf{$\ell=1$}.} We now consider the one copy case.
It suffices to show that $I$ is nothing but a Hecke translate of an irreducible component of $X_U$.
Suppose that the corresponding special subvariety does not equal an irreducible component of $X_{U}$.
It thus arises from a quaternion algebra over a proper subfield $E$ of $F$ (Lemma \ref{special1}).
In view of the definition of the CM points and the subset $\Delta_1$, it suffices to show that

\begin{equation}\label{gr}
\liminf_{M} \frac{|\Pic_{MF}|} {|\Pic_{M}|} = \infty                  .
\end{equation}
Here $M$ varies over CM quadratic extensions of the totally real field $E$.
Indeed the growth would imply that for $M$ with sufficiently large discriminant, typically a relevant ideal class of the CM extension $MF$
is not the inflation of an ideal class of $M$.
From CM theory, it would then follow that the CM points on $I$ corresponding to the ideal classes
eventually do not arise from the special subvariety.

As for the growth \eqref{gr}, it readily follows from class field theory and the Brauer--Siegel bounds (part (1) of Remark 1.2).
This finishes the proof.
\vskip2mm
{\bf{General $\ell$}.} We now suppose that $\ell \geq 2$.
Moreover, suppose that the corresponding special subvariety is proper i.e. the corresponding partition is non-trivial (Lemma \ref{special2}).
In particular, there exist $s\neq t$ such that
$$ \pi_{ s,t}  W  \subset X_{U}^{2}        $$
is the image of a Hecke correspondence.

We now regard the CM points as CM abelian varieties. 
Here we consider PEL Shimura variety arising from a base change of the quaternionic Shimura variety (\cite[\S2.7]{De2}).
It follows that the pair of CM points in the set $\pi_{s,t} T_{\ell}$ are isogenous to each other by an isogeny of a fixed degree.
Recall that the Brauer-Siegel lower bound implies that the class number $|C_{K}|$ increases with the discriminant of $K$.
This contradiction finishes the proof.
\end{proof}

\begin{remark}
Note that the proof slightly simplifies in the case $X$ being a Shimura curve, namely the initial step for $\ell=1$.
The case has applications to horizontal non-vanishing of Heegner points (\S3).
\end{remark}

For an application, we now consider a variant.
\begin{thm}\label{density2}
Let $F$ be a totally real field and $\fc$ an ideal of the integer ring $\cO_F$. 
Let $B$ be a quaternion algebra over $F$ such that it is not totally definite and $X_U$ the corresponding quaternionic Shimura variety with level $U$. 
Let $\Theta$ be an infinite set of CM quadratic extensions $K/F$ as above. Let
$\ell$ be a positive integer and
$$
\Xi_{\ell,\fc}=\{ (P_{K,\fc}^{\sigma_{i}\tau_{j}})_{1 \leq i,j \leq \ell}| \sigma_{i},\tau_{j} \in \Pic_{K/F}^{\fc}, K \in \Theta \}
$$
a subset of CM points in the
$\ell^{2}$-fold self-product $X_{U/\overline{\mathbb{Q}}}^{\ell^{2}}$. 

Then, $\Xi_{\ell,\fc}$ is Zariski dense.

\end{thm}
\begin{proof}
The argument is very similar to the proof of Theorem \ref{density1}. 

For simplicity of notation, we restrict to the case $\fc=1$ and let $\Xi_{\ell}=\Xi_{\ell,1}$. 
For $\ell=1$, the subset $\Xi_1$ is nothing but $\Delta_1$.
%the density is nothing but the previous theorem for $\ell=1$.
We thus suppose that $\ell\geq 2$.

Let $W' \subset X_{U}^{\ell^{2}}$ denote Zariski closure of the CM points $\Xi_{\ell}$.
In view of the previous theorem for $\ell=2$, the projection of $W'$ to any two factors of $X_{U}^{\ell^{2}}$ is dominant.

Let $I'$ be an irreducible component of $W'$.
It evidently contains an infinite subset $T_{\ell}'$ of $\Xi_{\ell}$.
The Andr\'{e}--Oort conjecture implies that
$I'$ is a special subvariety of the self-product $X_{U}^{\ell^{2}}$.
In view of Lemma \ref{special2}, we have an explicit list of the possibilities for $I'$.
Suppose that the subvariety is proper.
In particular, there exist $s\neq t$ and $m$ such that
$$ \pi_{ s,t}  T_{n}'  \subset Z_{m}        $$
with $|S_{m}|=2$.
This contradicts the dominance of $\pi_{s,t}$ and finishes the proof.
\end{proof}

\begin{remark}
(1). The proofs of Theorem \ref{density1} and Theorem \ref{density2} also indicate that variants of the Zariski density hold for `thin' subsets of CM points in the quaternionic case.
We restrict to the above versions as they suffice for the current horizontal non-vanishing. 

(2). In the case of a single copy of a quaternionic Shimura variety, we refer to \cite{Zh} and \cite{V} for the equidistribution of CM points in the complex analytic topology.
\end{remark}

\section{Non-vanishing of Heegner Points}
%\s{\bf Change later}: We may need to fix $K_{0, S}\subset B$ and choose embeddings of $K\in \Theta_1$ to $B$ carefully as mod p paper. Also, modify the proof of existence of toric test vectors.

In this section, we consider horizontal non-vanishing of Heegner points.

\subsection{Setup}
In this subsection, we introduce the setup.

Let $F$ be a totally real number field and $\BA$ its ring of \adeles . Let $\pi$ be a cuspidal automorphic representation of $\GL_2(\BA)$ with conductor $\fN$ and  a finite order central character $\omega$. Suppose that $\pi_v$ is a discrete series of weight $2$ for all $v|\infty$.
%Let $\phi\in S_2(\Gamma_0(N), \omega)$ be a newform of weight $2$, level $N$ and  Neben %character $\omega$. Let $\omega$ also denote its associated Hecke character on $\BA^\times$.
Let $\fc$ be an ideal of $\cO_{F}$ such that for $\fp |\fN$, we have
 $$\ord_\fp \omega_\fp\leq \ord_\fp \fc.$$ 
 Here $\omega_\fp$ denotes the local component of $\omega$.

Let $\BB$ be an incoherent quaternion algebra over $\BA$ such that there exists an irreducible automorphic  representation $\pi^{\BB}$ on $\BB^\times$ whose Jacquet--Langlands transfer is the automorphic representation $\pi$ of $\GL_2(\BA)$. 
Let $\BB_f$ denote the finite part.

Let $S=\Supp (\fN \fc \infty)$.  Let $K_{0, S}\subset \BB_S$ be a $F_S$-subalgebra such that 
\begin{itemize}
\item[(i)] $K_{0, \infty}=\BC$ and 
\item[(ii)] $K_{0, v}/F_v$ is semi-simple quadratic. 
\end{itemize}
For any $v\in S$, we say that $v$ is non-split if $K_{0, v}$ is a field and split otherwise. Let
$$ U_{0, S}:=\prod_{v\in S,\ \text{$v$  split}}\CO_{K_{0, v}}^\times \times \prod_{v\in S,\ \text{$v$ non-split}} K_{0, v}^\times.$$

Suppose we are given a finite order character $\chi_{0, S}: U_{0, S}\lra \ov{\BQ}^\times$ with conductor $\fc$ such that the following holds.
\begin{itemize}
\item[(LC1)]
$
\omega \cdot \chi_{0, S}\big|_{F_S^\times \cap U_{0, S}}=1.
$
 \item[(LC2)] $\prod_{v\in S} \chi_{0, v}(u)=1$ for all totally positive units $u$ in $F$. 

\item[(LC3)]
%We also suppose that
$
\epsilon(\pi, \chi_{0,v})\chi_{0,v}\eta_v(-1)=\epsilon(\BB_v)
$
for all places $v|\fN \fc \infty$ with the local root number $\epsilon(\pi,\chi_{0,v})$ corresponding to the Rankin--Selberg convolution. 
\end{itemize}
%We refer to $\chi_{0,S}$ as a local character with type $S$. 

Fix a maximal order $R^{(S)}$ of $\BB^{(S)}\cong M_2(\BA^{(S)})$.
Let $U^{(S)}=R^{(S)^\times}$. Note that $U^{(S)}$ is a maximal compact subgroup of
$\BB^{(S)^{\times}}\cong \GL_2(\BA^{(S)})$. 

We introduce the underlying CM quadratic extensions of the totally real field. 

\begin{defn}\label{Theta}
Let $\Theta_{S}$ denote the set of CM quadratic extensions $K/F$ such that
\begin{itemize}
%\item[(i)] $p\nmid |\Pic_{K, c}|$,
\item[(i)] there exists an embedding $\iota_{K}:\BA_{K}\ra \BB$ with $K_S=K_{0, S}$ and
\item[(ii)] $\BA_{K}^{(S)}\cap R^{(S)}=\wh{\CO}_K^{(S)}$ under the embedding.
\end{itemize}
\end{defn}
%There exist infinitely many imaginary quadratic extensions $K/\BQ$ with $K \in \Theta_S$
%(\cite{Br} and \cite{W}).
For $K \in \Theta_{S}$, we fix such an embedding $\iota_{K}$. 

We introduce the underlying Hecke characters over the CM quadratic extensions. 

\begin{defn}\label{char}
For each $K\in \Theta$, let $\fX_{K,\chi_{0,S}}$ denote the set of finite order Hecke characters $\chi$ over $K$ with $S$-type $\chi_{0,S}$ such that \begin{itemize}
\item[(i)] $\chi|_{\BA^\times} \cdot \omega=1$,
\item[(ii)] $\chi_S=\chi_{0, S}$ via the embedding $\iota_K$,  and
\item[(iii)] $\chi$ is unramified outside $S$.
\end{itemize}
\end{defn}
Note that the conductor of $\chi \in \fX_{K,\chi_{0,S}}$ equals $\fc$. Moreover, we have
\begin{itemize}
\item[(RN)] $\epsilon(\pi,\chi)=-1$.
\end{itemize}
Here $\epsilon(\pi,\chi)$ denotes the global root number of the Rankin--Selberg convolution corresponding to the pair $(\pi,\chi)$.

In the rest of the subsection, we let $\Theta=\Theta_{S}$ and
$\fX_{K} = \fX_{K,\chi_{0,S}}$.

\begin{lem}\label{existence1}
The set $\fX_{K}$ is non-empty for all but finitely many CM quadratic extensions $K/F$ with $K\in \Theta$. Moreover, it is a homogenous space for the relative class group $\Pic_{K/F}^{\fc}$.
\end{lem}
\begin{proof}
Note that for all but finitely many CM quadratic extensions $K/F$, we have that $\CO_K^\times=\CO_F^\times$. We fix such a CM quadratic extension from now. 
%For such a CM quadratic extension $K/F$, we now show that $\fX_K$ is 
%non-empty. 
%a homogenous space of $\Pic_{K/F}^\fc$.

In view of class field theory, the relative class group is given by
\begin{equation}\label{rcg}
\Pic_{K/F}^\fc=\BA_K^\times/K^\times K_\infty^\times \BA^\times U_\fc.
\end{equation}
%for $S_0$ the subset of non-split places.
%Suppose that $D_K\neq -3, -4$ so that  $\CO_K^\times=\BZ^\times$.
From the structure of $U_{\fc}$, there exists a Hecke character $\epsilon$ of $\BA_K^\times$ with  conductor $\fc$, trivial on $K_\infty^\times$ and $\BA^\times$ for all but finitely many CM quadratic extensions $K/F$ (for example \cite{KhKi}). 

Let $\chi_{0, S}'=\chi_{0, S} \cdot \epsilon|_{U_{0, S}}$. Then, there exists a character $\chi'$ of $\BA_K^\times/K^\times K_\infty^\times \BA^\times U_\fc$ extending $\chi_{0, S}'$. The character $\chi'\cdot \epsilon^{-1}$ is a desired one. This finishes the proof of first part. 

For $\chi \in \fX_K$, note that $\chi'\epsilon^{-1}\chi^{-1}$ factors through $\Pic_{K/F}^{\fc}$.
\end{proof}

%Let $\CO$ be the ring of integers of the field $\BQ(\pi, \chi_{0, S})$ generated over $\BQ$ by the Hecke eigenvalues of $\pi$ and the values of $\chi_{0, S}$.

Let $X$ be the Shimura curve over $F$ associated to $\BB^\times$. Recall that 
$X=\varprojlim_{U\subset \BB^\times_f} X_{U}$ for the Shimura curves $X_U$ with level $U$ for $U \subset \BB^{\times}_{f}$ an open compact subgroup.
Let $A_0$ be an abelian variety over $F$ corresponding to $\pi$. As in \cite{YZZ}, we have the representation of $\BB^\times$ over the field $M:=\End^0(A_0)$ given by
$$\pi^\BB=\varinjlim_{U\subset \BB^\times_f} \Hom_\xi^0(X_U, A_0).$$
Here $\xi$ is a Hodge class on $X$.
Recall that
$$
\pi^{\BB}\otimes_{M}\BC \simeq \pi_f
$$
for $\pi_f$ being the finite part of $\pi$.

Let $\CO$ be the ring generated over $\BZ$ by the image of $\chi_{0,S}$.  Let $A=A_0\otimes_\BZ \CO$ be the Serre tensor, which is endowed with endomorphisms
$\CO\subset \End_F(A)$.
Let $\pi^\BB\otimes_\BZ \CO=\varinjlim_U \Hom^0_\xi(X_U, A)$. 
Note that $\pi^\BB\otimes_\BZ \CO$  is a representation of $\BB^\times$ over $M\otimes_\BZ \CO$,
we call the scalar extension of $\pi^\BB$ (from $M$) to $M\otimes_\BZ \CO$.

We have the following existence of toric test vectors. 

\begin{lem}
\label{TV1}
There exists a non-zero form $f\in \pi^\BB$
%defined over $\cO_{(\fp)}$
satisfying the following.
\begin{itemize}
\item[(F1)] The subgroup $U_{0, S}$ acts on $f$ via $\chi_{0, S}$ and
\item[(F2)] $f \in \pi^{U^{(S)}}$.
\end{itemize}
\end{lem}

\begin{proof}
In view of (RN), there exists a non-zero $\ell \in \Hom_{\BA_{K}^{\times}}(\pi \otimes \chi, \BC)$ (\cite{S}, \cite{T0} and \cite{T}). 

From \cite{CST}, we thus have $\varphi \in \pi^{U^{(S)}}$ such that $\ell(\varphi)\neq 0$.
%We further choose $\varphi$ to be $\fp$-integral.
For $t \in U_{0,S}$, note that
$$
\ell(\chi_{0,S}(t)\varphi)=\chi_{0,S}(t)\varphi.
$$
Recall that
$$
\dim_{\BC}\Hom_{\BA_{K}^{\times}}(\pi \otimes \chi, \BC)=1.
$$

We can thus take $f$ to be the automorphic form given by
$$
g \mapsto \int_{U_{0,S}/F_{S}^{\times}\cap U_{0,S}} \chi_{0,S}(t)\varphi(gt) dt .
$$
%The existence follows from the existence of test vector corresponding to the pair $(\pi,\chi)$.
\end{proof}
Note that $f$ is spherical outside $S$.

%We further normalise $f$ to be $\fp$-primitive in the sense that the ideal generated by the image of $f$ is prime to $\fp$.

Let $U$ be an open compact subgroup of $\BB^{\times}$ such that 
\begin{itemize}
\item[(i)] $U=U_{S}U^{(S)}$ with $U_{S}\subset \BB_{S}^\times$, $U_{S}=\prod_{v \in S} U_{v}$ and $f\in \pi^U$;
\item[(ii)] $\CO_{\fc, v}^\times \subset U_v$ for all $v\in S$. 
Here $\cO_{\fc,v}=\cO_{0,v}+\fc\cO_{F,v}$ for $v \in S \backslash \supp\{\infty\}$. 
\end{itemize}
%Let $X$ be the corresponding Shimura curve of level
%$U$.

We identify $\BA_{K}^\times$ as a subgroup of $\BB^\times$ under the embedding $\iota_{K}$.
The choice of embedding $\iota_{K}$ thus gives rise to a map
$$
\varphi_{K,\fc}:\Pic_{K/F}^{\fc}\rightarrow X_{U/\ov{\BQ}}.
$$
For $\sigma \in \Pic_{K/F}^{\fc}$, let $x_{\sigma} \in X_{U/\ov{\BQ}}$ be the corresponding CM point on the Shimura curve.
%This is usually referred as a special point on the Shimura set.
In what follows, the CM points arising from 
$\varphi_{K,\fc}$ play an underlying role.

\begin{remark} 
In view of the consideration of a nearby of the incoherent quaternion algebra $\BB$, the construction of the Shimura curves and CM points as above is essentially equivalent to the constructions arising from a coherent quaternion algebra split at exactly one infinite place as in \S2.1 (\cite{YZZ} and \cite[Def. 2.24]{LZZ}).  
\end{remark}

\subsection{Non-vanishing}
In this subsection, we prove horizontal non-triviality of Heegner points and thus establish the  non-vanishing of central derivatives of Rankin--Selberg L-functions (Theorem \ref{derivative}). Our approach to the non-triviality is geometric.

\subsubsection{Main result} 
We describe the main result. 

Let the notation and assumptions be as in \S3.1.
Recall that the global root number of the Rankin--Selberg convolution for the pair $(\pi,\chi)$ equals $-1$. 
Accordingly, the central derivatives of the $L$-functions are expected to be generically non-vanishing (Conjecture \ref{NV}). 

Our main result regarding the non-vanishing is the following.

\begin{thm}\label{derivative} 
Let $F$ be a totally real number field, $\BA$ the ring of adeles and $\pi$ a cuspidal automorphic
representation of $\GL_{2}(\BA)$.
% with conductor $\frak{n}$.
Let $\fc \subset \cO_F$ be an ideal, $S=\supp(\fN\fc \infty)$ and $\chi_{0,S}$ a local character as above.
% with type $S$.
Let $\Theta_{S}$ be the set of CM quadratic extensions $K/F$ and $\fX_{K,\chi_{0,S}}$ the set of finite order Hecke characters over $K$ with $S$-type $\chi_{0,S}$ as above.
%Let $S$ be a finite set of places of $F$ containing all archimedean places and the finite places dividing $\frak{n}$.
%Let $S_1\subset S$ be the set of places $v|\frak{n}\infty$ and $K_{0, v}$ non-split.
%For $U_{0,S}$ as above, let $\chi_{0,S}$ be a character of $U_{0,S}$.
%Let $i\in \{0,1\}$ be such that
%$$\prod_{v\in S_1} \epsilon(1/2, \pi_v, \chi_{0, v})\chi_{0, v}\eta_v(-1)= (-1)^{i}.$$
%$$
%\epsilon(\pi,\chi_{0})=(-1)^{i}
%$$
%for some $\chi_{0} \in \fX_{K,\kappa,\fc}$.

For $\Theta \subset \Theta_{S}$ infinite, we have
$$\lim_{K\in \Theta} \#\left\{\chi\in \fX_{K,\chi_{0,S}} \ \Big|\ L'(1/2,\pi,\chi)\neq 0 \right\}=\infty.$$
\end{thm} 
The theorem will be proven in \S3.2.4 based on the Heegner introduced in \S3.2.2, control of torsion in \S3.2.3 and the Zariski density (Theorem \ref{density1}).

\subsubsection{Heegner points}
We introduce relevant Heegner points arising from toric test vectors in \S3.1 (\cite{YZZ}).

We first introduce a variant of CM points arising from $\Pic_{K/F}^{\fc}$ considered in \S3.1. In view of the Gross--Zagier formula, these CM points turn out to be the relevant ones. 

Let $$C_K=\BA_{K}^{(\infty),\times}/K^\times U_K$$ for $U_K$ arising from the embedding $\iota_K$ and the quaternionic level U in \S3.1
\footnote{Strictly speaking, $C_{K}$ also depends on $\fc$.}. The map $\varphi_{K,\fc}: \Pic_{K/F}^{\fc}\rightarrow X_{U/\ov{\BQ}}$ factors through $C_{K}$. 

Let $P_{K}\in X^{K^\times}$ be the CM point corresponding to the identity element in $C_{K}$. Here $K^\times$ acts on the Shimura curve $X$ via the embedding $\iota_{K}$. 
For $t \in C_K$, the corresponding CM point $P_{K}^{t}$ is defined over the ring class field $H_{K,\fc}$. Moreover, it is a Galois conjugate of the CM point $P_{K}$ via the Galois element 
$\sigma_{t}$ corresponding to $t$ via class field theory. 

We now introduce the Heegner points. 
\begin{defn}
Let $\chi \in \wh{C_{K}}$ be a finite order Hecke character over $K$ and $f$ a test vector as above. The Heegner point corresponding to the pair $(f,\chi)$ is given by
$$P_f(\chi):=\sum_{t\in C_K} f(P_{K})^{\sigma_t}\otimes \chi(t)\in A(\ov{\BQ}).$$
\end{defn}
By abuse of notation, let $H_{K, \fc}$ denote the ring class field of $K$ corresponding to $C_K$. Note that $P_{f}(\chi)\in A(H_{K, \fc})^{\chi}$ for the $\chi$-isotypic component $A(H_{K, \fc})^{\chi}$.

\begin{remark}
For $\chi \in \widehat{C_{K}}$, the non-vanishing of $P_{f}(\chi)$ implies that $\chi \in \fX_{K}$.
This follows from $f$ being $\chi_{0,S}$-toric.
The observation will be used in the proof of Theorem \ref{Heegner}.
\end{remark}

In view of the Gross-Zagier formula, we have 
$$L'(1/2, \pi, \chi)\neq 0 \iff P_f(\chi) \neq 0$$ 
(\cite{YZZ}). 
Thus, Theorem \ref{derivative} is equivalent to  the following 
\begin{thm}\label{Heegner} 
Let $F$ be a totally real number field, $\BA$ the ring of adeles and $\pi$ a cuspidal automorphic
representation of $\GL_{2}(\BA)$.
% with conductor $\frak{n}$.
Let $\fc \subset \cO_F$ be an ideal, $S=\supp(\fN\fc \infty)$ and $\chi_{0,S}$ a local character as above.
% with type $S$.
Let $\Theta_{S}$ be the set of CM quadratic extensions $K/F$ and $\fX_{K,\chi_{0,S}}$ the set of finite order Hecke characters over $K$ with $S$-type $\chi_{0,S}$ as above. 
For $\chi \in \fX_{K,\chi_{0,S}}$, let $P_{f}(\chi)$ be the Heegner point corresponding to the pair 
$(\pi,\chi)$ as above. 
%Let $S$ be a finite set of places of $F$ containing all archimedean places and the finite places dividing $\frak{n}$.
%Let $S_1\subset S$ be the set of places $v|\frak{n}\infty$ and $K_{0, v}$ non-split.
%For $U_{0,S}$ as above, let $\chi_{0,S}$ be a character of $U_{0,S}$.
%Let $i\in \{0,1\}$ be such that
%$$\prod_{v\in S_1} \epsilon(1/2, \pi_v, \chi_{0, v})\chi_{0, v}\eta_v(-1)= (-1)^{i}.$$
%$$
%\epsilon(\pi,\chi_{0})=(-1)^{i}
%$$
%for some $\chi_{0} \in \fX_{K,\kappa,\fc}$.

For $\Theta \subset \Theta_{S}$ infinite, we have

%such that $K_v\cong K_v^0$ for all $v\in S$ and $\fX_K\neq \varnothing$.
%Then
$$\lim_{K\in \Theta} \#\left\{\chi\in \fX_{K,\chi_{0,S}} \ \Big|\ P_{f}(\chi)\neq 0 \right\}=\infty.$$
\end{thm}

Our approach is based on 
\begin{itemize}
\item[--] Fourier analysis on the class group $C_K$ and its relation to CM points on self-products of the Shimura curve;
\item[--] Zariski density of well chosen CM points on self-products of the Shimura curve
and 
\item[--] finiteness of torsion points on the abelian variety over a compositum of class fields with bounded ramification.
\end{itemize}

\subsubsection{Torsion}

We prove a finiteness of torsion on a modular $\GL_2$-type abelian variety over a class of large extensions of the base field with bounded ramification.
\vskip2mm
{\it{Galois image}.} Let us first recall some results on the image of a $p$-adic Galois representation associated to a modular $\GL_2$-type abelian variety. 

Let $A$ be a modular $\GL_2$-type abelian variety and $\pi$ the corresponding cuspidal automorphic representation with central character $\omega$ as in \S3.1. For simplicity, we suppose that $A$ does not have CM.  
Let $L=\BQ(\pi)$ be the Hecke field and $F_{\omega}$ the extension cut out by the central character $\omega$.

Let $\Gamma \subset \Aut(L/\BQ)$ be the group of inner twists of $\pi$.
Recall that an inner twist of $\pi$ is a pair $(\sigma,\chi)$ with $\sigma$ an embedding of $L$ and $\chi$ a finite order character over $F$ such that there exists an isomorphism
$$^{\sigma}\pi \simeq \pi \otimes \chi$$
with $^{\sigma}\pi$ being the $\sigma$-conjugate. 
For such a $\sigma$, we have a unique Dirichlet character $\chi_{\sigma}$ such that the above isomorphism holds (\cite[B.3]{N}). 
Let $L^{\Gamma}$ be the fixed subfield of $L$ corresponding to $\Gamma$.
For $\sigma \in \Gamma$, we thus have an inner twist arising from $\chi_{\sigma}$.
Let $F_{\Gamma}/F$ be the extension corresponding to $\bigcap_{\sigma \in \Gamma}\ker(\chi_{\sigma})$.
Let $\frak{p}|p$ be a prime in $L$ and $\frak{p}_{\Gamma}$ the corresponding prime in $L^{\Gamma}$.

From the $\frak{p}$-adic Tate-module of $A$, we have the $\frak{p}$-adic Galois representation $\rho_{\frak{p}}: G_{F} \rightarrow \GL_{2}(O_{L,\frak{p}})$. It induces a Galois representation
$$\rho_\fp:  G_{F_{\omega}F_{\Gamma}}\lra \GL_{2}(O_{L^{\Gamma},\frak{p}_{\Gamma}}).$$
In view of the work of Ribet (\cite{Ri}) and Momose (\cite{Mo}), 
$$
\rho_{\frak{p}}(G_{F_{\omega}F_{\Gamma}}) \subseteq \bigg{\{}x \in \GL_{2}(O_{L^{\Gamma},\frak{p}_{\Gamma}}) \bigg{|} \det(x) \in \BZ_{p}^{\times} \bigg{\}}
$$
is an open subgroup (for example,  see \cite[B.5.2]{N}). Moreover, the equality holds for all but finitely many primes $\frak{p}$.
\vskip2mm
{\it{Finiteness}.} We have the following key proposition regarding control of torsion over ring class fields with bounded ramification. 
\begin{prop}
\label{torsion} 
Let $A$ be a modular $\GL_2$-type abelian variety over a  totally real field $F$. 

(1). Suppose that $A$ does not have CM. Then, there exist positive integers $n$ and $C$ such that for any CM quadratic extension $K/F$ the following holds. 
%which does not contains CM field of $A$ (if $A$ has CM), 

\begin{itemize}
\item For any prime $p>C$ and $0\neq Q\in A[p]$, we have $F(Q)\nsubseteq K^{ab}$ for the maximal abelian extension $K^{ab}$ of $K$. 
    \item For any prime $p$ and $Q\in A[p^n]\setminus A[p^{n-1}]$, we have $F(Q)\nsubseteq K^{ab}$.
        \end{itemize}
In particular, the set
$\displaystyle{\bigcup_K  A(H_{K, \fc})_\tor}$ is finite for an ideal $\fc$ of $\cO_F$ as $K$ varies over the CM quadratic extensions over $F$.
% and $H_{K, c}$ is the class field of $K$ corresponding to $C_K$. 

(2). If $A$ has CM, the same conclusion holds for the variation in the first two assertions over CM quadratic extensions of $F$ which do not contain the CM fields arising from the endomorphism algebra $\End^{\BQ}_{F}(A)$. Moreover, `In particular' part holds for variation over all the CM quadratic extensions of $F$. 
\end{prop}
\begin{proof}The approach is based on the Galois image result. 

First we consider the case when $A$ does not have CM. Replacing $A$ by an isogeny, we may assume $\CO\subset \End_F(A)$ for $\CO$ being the ring of integer of $F'=F_\omega F_\Gamma$.  
Let $C$ be an integer such that 
\begin{itemize}
\item[--] for any prime $\ell >C$ and any prime ideal $\fp$ of $\CO$ above $\ell$, the equality of the Galois image alluded to earlier holds and 
\item[--] $\PSL_2(k)$ is simple where $k$ is the residue field of $\fp$. 
\end{itemize}

Let $K$ be a CM quadratic extension of $F$ and $K'=KF'$. Then, the extension $K'(A[\fp])/F'$ is a Galois extension. If there exists a non-trivial torsion point $Q\in A[\fp]$ contained in $A(K^{ab})$, then 
%one can see that 
$K'(A[\fp])/K'(Q)$ is a two step abelian extension. Moreover, $K'(P)/K'$ is abelian by assumption. It follows that $K'(A[\fp])/F'$ is a solvable extension and the same holds for $F'(A[\fp])/F'$. On the other hand, 
$$\Gal(F'(A[\fp])/F')=\bigg{\{} x\in \GL_2(k)\ \bigg{|}\ \det x\in \BF_p^\times \bigg{\}}.$$
Thus,  $\Gal(F'(A[\fp])/F')$ has 
$\PSL_2(k)$ as its sub-quotient. 
This contradiction finishes the proof of first assertion. 

For any finite solvable group $G$, let $\mu(G)$ be the smallest non-negative integer $n$ such that there exists a sequence $$G=G_0\supset G_1\supset \cdots \supset G_n=1$$ with $G_i/G_{i+1}$ abelian.  For a prime $p$ and  a prime ideal $\fp$ of $F$ above $p$, suppose the second statement is not true. 
In other words, we then have $$\mu(K'(A[\fp^n])/F')\leq 4$$ for each $n$. 
In view of a simple estimation property of $\mu$, 
this is a contradiction to the Galois image result and the fact that
$$\lim_{n\ra \infty}  \mu(1+\fp M_2(\CO/\fp^n))=\infty.$$

For the case when $A$ has CM, the statement can be analogously be approached via the Galois image result (see \cite[Thm. B.6.3]{N}).

\end{proof}
\begin{remark}
In the non-CM case, the proof shows finiteness of the torsion over a class of large solvable extensions.
\end{remark}

\subsubsection{Non-vanishing}
We prove horizontal non-vanishing of Heegner points (Theorem \ref{Heegner}).
\vskip2mm
{\it{Bounded annihilator}.} Let us first recall the the following elementary lemma from \cite{MV1}.
\begin{lem}\label{lemMV} Let $G$ be a finite abelian group and $W$ a $k$-dimensional $\BQ$-representation of $G$. Let $S\subset G$ be a subset with $|S|=2k$. 

Then, there exists integers $n_s\in \BZ$, not all zero, such that the element $\sum_{s\in S} n_s s\in \BZ[S]$ annihilates the $\BQ[G]$-module $W$. Moreover, we may choose $n_s$ so that $$|n_s| \ll D^{k^2}$$ for an absolute constant $D$ 
(\cite[Lem. 4.2]{MV1}).
\end{lem}
\vskip2mm
{\it{Proof of Theorem \ref{Heegner}}.} 

The modular parametrisation $f$ gives rise to a function on $C_{K}$ with values in the Mordell--Weil group $A(H_{K,\fc})$ for the ring class field $H_{K,\fc}$ (\S3.3.2).

In view of Fourier inversion on $C_K$ and Proposition \ref{torsion}, we note
$$
\liminf_K \#\bigg{\{}\chi\in \fX_{K,\chi_{0,S}}\bigg{|} P_f(\chi)\neq 0\bigg{\}}>0.
$$
Otherwise, $f: C_{K} \ra A(H_{K,\fc})_{\BQ}$ would be a constant function for any $K$ as above with sufficiently large discriminant. In view of Proposition \ref{torsion} and the Zariski density of the CM points arising from $C_{K}$, the $K$-constancy would contradict the modular parametrisation $f$ being a non-constant morphism. 

We now turn to the growth in the number of non-vanishing twists.
% (Theorem \ref{Heegner}), 
%we need 
%We now return to the proof of Theorem \ref{Heegner}.
Suppose that the theorem does not hold. 
In other words, there exists an infinite subset $\Theta_1$ of $\Theta$ and an integer $k\geq 1$ such that for any $K\in \Theta_1$ 
\begin{equation}\label{bd}
\# \bigg{\{} \chi\in \fX_{K,\chi_{0,S}} \ \bigg{|}\ P_f(\chi)\neq 0\bigg{\}} =k.
\end{equation}

Let $W_{K}=\langle f(P)^{\sigma_t} | t \in C_{K} \rangle_{\BQ}$ be the subspace of 
$A(H_{K,\fc})_{\BQ}=A(H_{K,\fc})\otimes_\BZ \BQ$. 
In view of the hypothesis \eqref{bd}, we have
$$\dim_{\BQ} W_{K}=k.$$ 
By definition, $W$ is a $\BQ[C_K]$-module. Moreover, the action relates to Shimura's reciprocity law regarding the Galois action on CM points. 
% $t\in C_K=\wh{K}/K^\times U_K$, generate a $k$-dimensional subspace 
%Let $W$ of
%$A(H_{K,c})_{\BQ}=A(H_{K,c})\otimes_\BZ \BQ$ (Note that $f$ is $\chi_{0, S}$-invariant under $U_{0, S}$). 

Let $m=2k$ and consider the map
$$h=f^m \circ g: C_{K}^m \stackrel{g}{\lra} X_U^m \stackrel{f^m}{\lra} A^m.$$
Here the map $g$ is given by $g(t_1, \cdots, t_m)=(P_{K}^{\sigma_{t_1}}, \cdots, P_{K}^{\sigma_{t_m}})$ for any $(t_1, \cdots, t_m)\in C_{K}$ and $f^m$ arises from the modular parametrisation $f: X_{U} \ra A$.
It follows that 
$$h(C_{K}^m)=
\bigg{\{}(f(P_{K})^{\sigma_{t_1}}, \cdots, f(P_{K})^{\sigma_{t_m}}\big)\in A(H_{K, \fc})^m_\BQ \bigg{|} (t_1, \cdots, t_m)\in C_{K}^m\bigg{\}}$$
for all $K\in \Theta_1$.

In view of Lemma \ref{lemMV}, there exist integers $(n_i)\in \BZ^m$ with $|n_s|<D$ such that 
for any $(t_1, \cdots, t_m)\in C_{K}^m$, we have

$$
\displaystyle{\sum_{i=1}^m n_i f(P_{K})^{\sigma_{t_i}}=0 \ \text{in}\ A(H_{K, \fc})_\BQ.}
$$
In particular, any point $(Q_1, \cdots, Q_m)$ in the image of $h$ satisfies one of the following relations:
$$\sum_{i=1}^m n_i Q_i+T=0, \qquad (n_i)\in [-D, D]^m\cap \BZ^m, \qquad  T\in \bigcup_K A(H_{K, \fc})_\tor.$$ 

In view of Proposition \ref{torsion}, the number of such relations is finite as $K$-varies. 

On the other hand, from the Zariski density in Theorem \ref{density1}, the image of $g$ is dense in $X_U^m$. We conclude that
 $$\Im (h) \subset f^m(X_{U}^m)=f(X_{U})^m $$
 is Zariski dense. 
 Thus, $f(X_{U})^m$ is contained in the union of the finitely many ``hyperplanes" in $A^m$ as above. 
 
  This is a contradiction (for example, there exists $(y_1 , \cdots, y_n) \in \Im(h)$ with $y_1, \cdots, y_{n-1}$ algebraic and $y_n$ transcendental). 
  
  This finishes the proof of Theorem \ref{Heegner}.

\begin{remark}
(1). The proof crucially relies on the existence of a bounded annihilator (Lemma \ref{lemMV}). 
In the approach to the horizontal non-vanishing of Heegner points due to Michel--Venkatesh (\cite{MV1}), 
curiously such an annihilator also appears in an intermediate step. 

(2). One can ask for a higher weight analogue of Theorem \ref{derivative}. Namely, we may fix $\pi$ with weight at least two and vary $\chi$ with fixed infinity type such that the Rankin--Selberg convolution associated to the pair $(\pi,\chi)$ is self-dual with root number $-1$. Modulo conjectures on the non-vanishing of the height of non-torsion cycles, Gross--Zagier formula would reduce the non-vanishing to the non-triviality of generalised Heegner cycles on a Kuga--Sato variety over a Shimura curve. The Kuga--Sato variety typically depends on the weight and the CM extension. When $\chi$ is of finite order, the Kuga--Sato variety 
%$\mathcal{X}$ 
only depends on the weight. In this case, we may ask whether the current approach applies. We hope to consider the setup elsewhere.

%The finiteness of torsion in the middle dimensional Chow group of $\mathcal{X}$ does not seem immediate. Even though we have Galois image results, some information on the kernel of \'{e}tale Abel-Jacobi map seems to be necessary (the case of $3$-folds is perhaps fine).

%A further key ingredient would be Zariski density of generalised Heegner cycles. This seems to follow from positive dimensional Andr\'{e}-Oort conjecture for Kuga--Sato varieties. The conjecture does not seem to be known in general (again the case of $3$-folds is perhaps fine).
\end{remark}

\section{Non-vanishing of toric periods}
%\s{\bf Change later}: We may need to fix $K_{0, S}\subset B$ and choose embeddings of $K\in \Theta_1$ to $B$ carefully as mod p paper. Also, modify the proof of existence of toric test vectors.
%Add a brief discussion of Eisenstein series case and slightly rewrite the definite case.

In this section, we consider horizontal non-vanishing of toric periods.
\subsection{Indefinite case}
In this subsection, we consider the non-vanishing of toric periods when the underlying quaternion algebra is not totally definite.
\subsubsection{Setup}
We describe the setup and main result. The result will be proven in the following subsections.

Let $F$ be a totally real number field and $\BA$ its ring of \adeles. 
Let $I$ denote the set of infinite places of $F$. 
Let $\pi$ be a cuspidal automorphic representation of $\GL_2(\BA)$ with conductor $\fN$ and  a finite order central character $\omega$. Suppose that $\pi_\infty$
is a discrete series for $\GL_2(F_\infty)$ of weight $(k_v)_{v|\infty}$ with the same parity. 
%Let $k_0=\max_v k_v$.
%Let $\phi\in S_2(\Gamma_0(N), \omega)$ be a newform of weight $2$, level $N$ and  Neben %character $\omega$. Let $\omega$ also denote its associated Hecke character on $\BA^\times$.
Let $\fc$ be an ideal of $\cO_{F}$ such that for $\fp |\fN$, we have 
$$\ord_\fp \omega_\fp\leq \ord_\fp \fc.$$ Here $\omega_\fp$ denotes the local component of $\omega$.

Let $B$ be quaternion algebra over $F$ which is not totally definite such that there exists an irreducible automorphic  representation $\pi^{B}$ on $B_{\BA}^\times$ whose Jacquet--Langlands transfer is the automorphic representation $\pi$ of $\GL_2(\BA)$. 
Here $B_{\BA}=B\otimes_{F} \BA$. Let $I_{B} \subset I$ be a subset such that $\sigma \in I_{B}$ if and only if $B$ splits at $\sigma$. Let $I^{B}$ denote the complement of $I_{B}$. 

For a CM quadratic extension $K/F$ with CM type $\Sigma$, we often identify $\Sigma$ with $I$. 
%Let . 
For an integer $l \geq 0$ and $\kappa \in \BZ_{\geq 0}[\Sigma]$, 
let $l\Sigma + \kappa(1-c)$ be an allowable infinity type for arithmetic Hecke characters over a CM quadratic extension $K/F$ with $c \in \Gal(K/F)$ the non-trivial element. Suppose that 
\begin{itemize}
\item[(ID)] for $\sigma \in I_{B}$, $l+2\kappa_{\sigma} \geq k_{\sigma}$ and 
for $\tau \in I^{B}$, $l+2\kappa_{\tau} < k_{\tau}$.
\end{itemize}
%with all
%$\kappa_v\in \BZ_{\geq 1-k_{0}}$ and at least one $\kappa_v\geq 0$.

Let $S=\Supp (\fN \fc \infty)$.  Let $K_{0, S}\subset B_S$ be a $F_S$-subalgebra such that 
\begin{itemize}
\item[(i)] $K_{0, \infty}=\BC$ and 
\item[(ii)] $K_{0, v}/F_v$ is semi-simple quadratic. 
\end{itemize}
For any $v\in S$, we say that $v$ is non-split if $K_{0, v}$ is a field and split otherwise as before and let
$$ U_{0, S}:=\prod_{v\in S,\ \text{$v$  split}}\CO_{K_{0, v}}^\times \times \prod_{v\in S,\ \text{$v$ non-split}} K_{0, v}^\times.$$

Suppose we are given a character $\chi_{0, S}: U_{0, S}\lra \BC^\times$ with conductor $\fc$ such that the following holds.
\begin{itemize}
\item[(LC1)]
$
\omega \cdot \chi_{0, S}\big|_{F_S^\times \cap U_{0, S}}=1.
$
 \item[(LC2)] The archimedean-component $\chi_{0,\infty}$ is with infinity type $l\Sigma+\kappa(1-c)$. %satisfying the hypothesis (ID).
 %For $v|\infty$, we have
 %$\chi_{0, v}(z)=(z/\ov{z})^{\frac{k_v}{2}+\kappa_v}$.

 \item[(LC3)] $\prod_{v\in S} \chi_{0, v}(u)=1$ for all totally positive units $u$ in $F$.

\item[(LC4)]
%We also suppose that
$
\epsilon(\pi, \chi_{0,v})\chi_{0,v}\eta_v(-1)=\epsilon(B_v)
$
for all places $v|\fN \fc \infty$ with the local root number $\epsilon(\pi,\chi_{0,v})$ corresponding to the Rankin--Selberg convolution. 
\end{itemize}

% \begin{remark}
% Note that for $v|\infty$, the theta series associated to a Hecke character with local-type 
% $\chi_{0, v}$ is of weight $k_v+2\kappa_v+1$. If $-k<\kappa<0$, then we will choose quaternion
% ramified at $v$ and if $\kappa\geq 0$, then split at $v$. In the case $\kappa\geq 0$, the test vector
% is of weight $k_v+2\kappa_v$, i.e. the minimal holomorphic vector acts by the weight-lifting operator $\kappa_v$
% times. 
 %\end{remark}

 Fix a maximal order $R^{(S)}$ of $B_{\BA}^{(S)}\cong M_2(\BA^{(S)})$.
Let $U^{(S)}=R^{(S)^\times}$. Note that $U^{(S)}$ is a maximal compact subgroup of
$B_{\BA}^{(S)^{\times}}\cong \GL_2(\BA^{(S)})$. 

We introduce the underlying CM quadratic extensions of the totally real field.

 \begin{defn}\label{Theta'}
Let $\Theta_{S}$ denote the set of CM quadratic extensions $K/F$ such that
\begin{itemize}
%\item[(i)] $p\nmid |\Pic_{K, c}|$,
\item[(i)] there exists an embedding $\iota_{K}:K\ra B$ with $K_S=K_{0, S}$ and
\item[(ii)] $\BA_{K}^{(S)}\cap R^{(S)}=\wh{\CO}_K^{(S)}$ under the embedding.
\end{itemize}
\end{defn}
%There exist infinitely many imaginary quadratic extensions $K/\BQ$ with $K \in \Theta_S$
%(\cite{Br} and \cite{W}).
For $K \in \Theta_{S}$, we fix such an embedding $\iota_{K}$.

We introduce the underlying Hecke characters over the CM quadratic extensions. 
\begin{defn}\label{char'}
For each $K\in \Theta$, let $\fX_{K,\chi_{0,S}}$ denote the set of characters $\chi$ over $K$ with $S$-type $\chi_{0,S}$ such that \begin{itemize}
\item[(i)] $\chi|_{\BA^\times} \cdot \omega=1$,
\item[(ii)] $\chi_S=\chi_{0, S}$ via the embedding $\iota_K$,  and
\item[(iii)] $\chi$ is unramified outside $S$.
\end{itemize}
\end{defn}
Note that the conductor of $\chi \in \fX_{K,\chi_{0,S}}$ equals $\fc$. Moreover, we have
\begin{itemize}
\item[(RN)] $\epsilon(\pi,\chi)=1$.
\end{itemize}
Here $\epsilon(\pi,\chi)$ denotes the global root number of the Rankin--Selberg convolution corresponding to the pair $(\pi,\chi)$.

In the rest of the section, we let $\Theta=\Theta_{S}$ and
$\fX_{K} = \fX_{K,\chi_{0,S}}$.

\begin{lem}\label{existence1'}
The set $\fX_{K}$ is non-empty for all but finitely many CM quadratic extensions $K/F$ with $K\in \Theta$. Moreover, it is a homogenous space for the relative class group $\Pic_{K/F}^{\fc}$.
\end{lem}
 \begin{proof}
 The same argument as in the proof of Lemma \ref{existence1} applies.
 \end{proof}

% Let $S_0\subset S$ be the set of places $v|N\infty$ and $K_{0, v}$ non-split. Assume that
%$$\prod_{v\in S_0} \epsilon(1/2, \pi_v, \chi_{0, v})\chi_{0, v}\eta_v(-1)=+1.$$

%For each quadratic field extension $K$ over $F$ such that for all $v\in S$, $K_v\cong K_{0, v}$ (fix such an isomorphism once for all and identify $K_v$ with $K_{0, v}$),  let $\fX_K$ denote the group of Hecke characters $\chi: K^\times\bs K_\BA^\times\ra \BC^\times$ satisfying that
%\begin{itemize}
%\item[(1)] $\chi$ is unramified outside $S$,
%\item[(2)] $\chi_v$ is equal to $\chi_{0, v}$ on $\CO_{K_v}^\times$ (resp. on $K_v^\times$) for all $v\in S$ split (resp. non-split) in $K$,
%\item[(3)] $\omega\cdot \chi|_{\BA^\times}=1$.
%\end{itemize}

We have the following existence of toric test vectors.

\begin{lem}
\label{TV1'}
There exists a non-zero form $f\in \pi^B$
%defined over $\cO_{(\fp)}$
satisfying the following.
\begin{itemize}
\item[(F1)] The subgroup $U_{0, S}$ acts on $f$ via $\chi_{0, S}$ and
\item[(F2)] $f \in \pi^{U^{(S)}}$.
\end{itemize}
\end{lem}
\begin{proof}
The same argument as in the proof of Lemma \ref{TV1} applies.
\end{proof}

Note that $f$ is spherical outside $S$.

%We further normalise $f$ to be $\fp$-primitive in the sense that the ideal generated by the image of $f$ is prime to $\fp$.

Let $U$ be an open compact subgroup of $B_{\BA}^{(\infty),\times}$ such that 
\begin{itemize}
\item[(i)] $U=U_{S}U^{(S)}$ with $U_{S}\subset B_{S}^\times$, $U_{S}=\prod_{v \in S} U_{v}$ and $f\in \pi^U$; 
\item[(ii)] $\CO_{c, v}^\times \subset U_v$ for all $v\in S \backslash \supp\{\infty\}$. 
Here $\cO_{\fc,v}=\cO_{0,v}+\fc\cO_{F,v}$ for $v \in S$. 
\end{itemize}

Let $X_U$ be the corresponding quaternionic Shimura variety of level
$U$ (\S 2.1).

In view of the theory of quaternionic modular forms (\cite{HMF} and \cite{Ti}), we recall that $f$ admits algebro-geometric interpretation. The interpretation plays a key role in the following subsection.

We identify $\BA_{K}^\times$ as a subgroup of $B_{\BA}^\times$ under the embedding $\iota_{K}$.
The choice of embedding $\iota_{K}$ thus gives rise to a map
$$
\varphi_{K}:\Pic_{K/F}^{\fc}\rightarrow X_{U/\ov{\BQ}}
$$
(\S 2.2). 
For $\sigma \in \Pic_{K/F}^{\fc}$, let $x_{\sigma} \in X_{U/\ov{\BQ}}$ be the corresponding CM point on the quaternionic Shimura variety.
%This is usually referred as a special point on the Shimura set.
In what follows, the map
$\varphi_{K}$ plays an underlying role.

\subsubsection{Main result} 
We describe the main result. 

Let the notation and assumptions be as in \S4.1.1.
Recall that the global root number of the Rankin--Selberg convolution for the pair $(\pi,\chi)$ equals $1$. Accordingly, the central values of the $L$-functions are expected to be generically non-vanishing (Conjecture \ref{NV}). 

Our main result regarding the non-vanishing is the following.

\begin{thm}\label{L-value} 
Let $F$ be a totally real number field, $\BA$ the ring of adeles and $\pi$ a cuspidal automorphic
representation of $\GL_{2}(\BA)$.
% with conductor $\frak{n}$.
Let $\fc \subset \cO_F$ be an ideal, $S=\supp(\fN\fc \infty)$ and $\chi_{0,S}$ a local character as above.
% with type $S$.
Let $\Theta_{S}$ be the set of CM quadratic extensions $K/F$ and $\fX_{K,\chi_{0,S}}$ the set of Hecke characters over $K$ with $S$-type $\chi_{0,S}$ as above.
%Let $S$ be a finite set of places of $F$ containing all archimedean places and the finite places dividing $\frak{n}$.
%Let $S_1\subset S$ be the set of places $v|\frak{n}\infty$ and $K_{0, v}$ non-split.
%For $U_{0,S}$ as above, let $\chi_{0,S}$ be a character of $U_{0,S}$.
%Let $i\in \{0,1\}$ be such that
%$$\prod_{v\in S_1} \epsilon(1/2, \pi_v, \chi_{0, v})\chi_{0, v}\eta_v(-1)= (-1)^{i}.$$
%$$
%\epsilon(\pi,\chi_{0})=(-1)^{i}
%$$
%for some $\chi_{0} \in \fX_{K,\kappa,\fc}$.

For $\Theta \subset \Theta_{S}$ infinite, we have

$$\lim_{K\in \Theta} \#\left\{\chi\in \fX_{K,\chi_{0,S}} \ \Big|\ L(1/2,\pi,\chi)\neq 0 \right\}=\infty.$$
\end{thm}

The theorem will be proven in \S4.1.4 based on the the toric periods in \S4.1.3 and the Zariski density (Theorem \ref{density2}).

\subsubsection{Toric periods}
We introduce relevant toric periods arising from toric test vectors in \S4.1.1(\cite{YZZ}).

We first introduce a variant of CM points arising from $\Pic_{K/F}^{\fc}$ considered in \S4.1.1. In view of the Waldspurger formula, these CM points turn out to be the relevant ones. 

Let $$C_K=\BA_{K}^{(\infty),\times}/K^\times U_K$$ for $U_K$ arising from the embedding $\iota_K$ and the quaternionic level U in \S4.1.1. The map $\varphi_{K,\fc}: \Pic_{K/F}^{\fc}\rightarrow X_{U/\ov{\BQ}}$ factors through $C_{K}$.

%We fix a $\chi_{0} \in \fX_{K,\chi_{0,S}}$

\begin{defn}
Let $\chi \in \fX_{K,\chi_{0,S}}$ and $f$ a test vector as above. Then, the toric period corresponding to the pair $(f,\chi)$ is given by
$$P_f(\chi):=\sum_{\sigma \in C_K} f(x_{\sigma})\chi(\sigma).$$
\end{defn}

\begin{remark}
If we fix a $\chi_{K} \in \fX_{K,\chi_{0,S}}$, then the toric period $P_{f}(\chi\chi_{K})$ 
is well defined for any $\chi \in \wh{C_{K}}$.
However, the non-vanishing of $P_{f}(\chi\chi_{K})$ implies that $\chi\chi_{K} \in \fX_{K}$.
This follows from $f$ being $\chi_{0,S}$-toric.
The observation will be used in the proof of Theorem \ref{period}.
\end{remark}
%\begin{remark}
%For $\chi \in \widehat{C_{K}}$, the non-vanishing of $P_{f}(\chi)$ implies that $\chi \in \fX_{K}$.
%This follows from $f$ being $\chi_{0,S}$-toric.
%The observation will be used in the proof of Theorem \ref{period}.
%\end{remark}
In view of the Waldspurger formula, we have that 
$$
L(1/2, \pi, \chi)\neq 0 \iff P_f(\chi) \neq 0
$$ 
(\cite{YZZ}). Thus, Theorem \ref{L-value} is equivalent to  the following 
\begin{thm}\label{period}
%Let the notation and assumptions be as above.
Let $F$ be a totally real number field, $\BA$ the ring of adeles and $\pi$ a cuspidal automorphic
representation of $\GL_{2}(\BA)$.
% with conductor $\frak{n}$.
Let $\fc \subset \cO_F$ be an ideal, $S=\supp(\fN\fc \infty)$ and $\chi_{0,S}$ a local character as above.
% with type $S$.
Let $\Theta_{S}$ be the set of CM quadratic extensions $K/F$ and $\fX_{K,\chi_{0,S}}$ the set of Hecke characters over $K$ with $S$-type $\chi_{0,S}$ as above. 
For $\chi \in \fX_{K,\chi_{0,S}}$, let $P_{f}(\chi)$ be the toric period as above. 
%Let $S$ be a finite set of places of $F$ containing all archimedean places and the finite places dividing $\frak{n}$.
%Let $S_1\subset S$ be the set of places $v|\frak{n}\infty$ and $K_{0, v}$ non-split.
%For $U_{0,S}$ as above, let $\chi_{0,S}$ be a character of $U_{0,S}$.
%Let $i\in \{0,1\}$ be such that
%$$\prod_{v\in S_1} \epsilon(1/2, \pi_v, \chi_{0, v})\chi_{0, v}\eta_v(-1)= (-1)^{i}.$$
%$$
%\epsilon(\pi,\chi_{0})=(-1)^{i}
%$$
%for some $\chi_{0} \in \fX_{K,\kappa,\fc}$.

For $\Theta \subset \Theta_{S}$ infinite, we have

$$\lim_{K\in \Theta} \#\left\{\chi\in \fX_{K,\chi_{0,S}} \ \Big|\ P_{f}(\chi)\neq 0 \right\}=\infty.$$
\end{thm}

Our approach is based on 
\begin{itemize}
\item[--] Fourier analysis on the class group $C_K$ and its relation to CM points on self-products of the quaternionic Shimura variety; 
\item[--] Zariski density of well chosen CM points on self-products of the quaternionic Shimura variety. 
\end{itemize}

\subsubsection{Non-vanishing}
We prove the horizontal non-vanishing of toric periods (Theorem \ref{period}). The approach is based on the Hilbert modular case in \cite{BH}. 
%We keep the exposition brief. 

We begin with the proof. 

Suppose that the statement is not true. In other words, there exists an infinite subset $\Theta_1\subset \Theta$ and an integer $k\geq 1$ such that for any $K\in \Theta_1$
\begin{equation}\label{bd'}
\bigg{\{}\chi\in \fX_K\ \bigg{|}\ P_f(\chi)\neq 0\bigg{\}}=k-1.
\end{equation}

We now choose a $\chi_{K} \in \fX_{K}$.
Note that $f\cdot \chi_K$ is a well-defined $\ov{\BQ}$-valued function on $C_K$ for each $K$.

For $K \in \Theta$, let $t_{1}, ..., t_{k} \in C_{K}$ be $k$ distinct elements.
In view of the hypothesis \eqref{bd'}, it follows that the functions given by
$$
t \mapsto f(tt_{i})\chi_{K}(tt_{i})
$$
viewed as elements in the vector space of maps $C_{K} \rightarrow \overline{\BQ}$ are linearly dependent for $1 \leq i \leq k$.
Say,
$$
\sum_{i=1}^{k} c_{K,i}f(tt_{i})\chi_{K}(tt_{i}) = 0
$$
for some $c_{K,i} \in \overline{\BQ}$ and any $t \in C_{K}$.

We now choose $k$ distinct elements $s_{1},...,s_{k} \in C_{K}$ and consider the above dependence  for
$t=s_{1},...,s_{k}$. It follows that
$$
\det\bigg{(}f(s_{j}t_{i})\chi_{K}(s_{j}t_{i})\bigg{)}_{1\leq i,j \leq k}=0.
$$
Note that
$$
\bigg{(}f(s_{j}t_{i})\chi_{K}(s_{j}t_{i})\bigg{)}_{1\leq i,j \leq k}=
\diag\bigg{(}\chi_{K}(s_{j})\bigg{)}_{1\leq j \leq k}
\bigg{(}f(s_{j}t_{i})\bigg{)}_{1\leq i,j \leq k}
\diag\bigg{(}\chi_{K}(t_{i})\bigg{)}_{1\leq i \leq k}.
$$
Here $\diag\big{(}a_{i}\big{)}_{1\leq i \leq k}$ denotes the diagonal matrix with diagonal entries
$\{ a_{1}, ..., a_{k} \}$.

We conclude that
$$
\det\bigg{(}f(s_{j}t_{i})\bigg{)}_{1\leq i,j \leq k}=0.
$$

We now consider the function $f_{k}$ on the self-product $X_{U/\ov{\BQ}}^{k^{2}}$ given by
$$
\big{(}x_{i,j}\big{)}_{1\leq i,j \leq k} \mapsto
\det\bigg{(}f(x_{i,j})\bigg{)}_{1\leq i,j \leq k}.
$$
The lower triangular entries of the above matrix can be all arranged to be zero 
\footnote{In view of the algebro-geometric interpretation of quaternionic modular forms, $f$ has a zero.}.
Moreover, the product of the diagonal entries can be arranged to be non-constant simultaneously as
$f$ a non-constant morphism.
It follows that $f_{k}$ is a non-constant function.

The function $f_{k}$ vanishes on the collection $\Xi_{k}$ of CM points.
This contradicts the density in Theorem \ref{density2} and finishes the proof. 

\begin{remark}
(1). The proof crucially relies on the Zariski density (Theorem \ref{density1}). 

(2). Under certain hypothesis, the results were announced in \cite[\S4]{BH}.

(3). Based on the current approach, an analogous result can be proven for special values of Hecke L-functions. We may consider Hecke characters with a fixed infinity type and bounded conductor as CM quadratic extensions of a fixed totally real field vary. 
\end{remark}

%\end{proof}
%\begin{remark}
%The above argument is generalisation of the Hilbert modular consideration in \cite{BH}.
%\end{remark}

%\subsection{\bf Eisenstein Series Case}

\subsection{Definite case}
In this subsection, we consider the non-vanishing of toric periods when the underlying quaternion algebra is totally definite. 

We follow the approach in \cite{MV} rather closely and keep the exposition brief. 

Let the setup be as in \S4.1.1 with the only exception that $B$ is a totally definite quaternion algebra over the totally real field $F$. In particular, the set $\Theta_S$ denotes the set of CM quadratic extensions $K/F$ admitting an embedding $\iota_{K}: K \hookrightarrow B$ and $\fX_{K,\chi_{0,S}}$ denotes the set of Hecke characters over $K$ with the local type being $\chi_{0,S}$. 

Moreover, we suppose that $\pi$ is with parallel weight two. In particular, $\fX_{K,\chi_{0,S}}$ 
consists of finite order Hecke characters over $K$. 

The existence of toric vector (Lemma \ref{TV1'}) goes through verbatim. We use the same notation $f$ for a choice of a toric vector and $U$ for the quaternionic level. For $\chi \in \fX_{K,\chi_{0,S}}$, let $P_{f}(\chi)$ denote the corresponding toric period. 

In this setup, an underlying feature is that the underlying double coset $X_U$ no longer corresponds to an algebraic variety. In fact, $X_U$ is a finite set of points usually referred as a definite Shimura set 
\footnote{or a Hida `variety'}. 
Nevertheless, the embedding $\iota_{K}$ still gives rise to a map
$$
\varphi_{K}:\Pic_{K/F}^{\fc}\rightarrow X_{U}.
$$
For $\sigma \in \Pic_{K/F}^{\fc}$, let $x_{\sigma} \in X_{U}$ be the corresponding special point on the definite Shimura set.

%Recall that the global root number of the Rankin-Selberg convolution for the pair $(\pi,\chi)$ equals $1$. Accordingly, the central values of the $L$-functions are expected to be generically non-vanishing (Conjecture \ref{NV}). 
Our main result regarding the non-vanishing is the following

\begin{thm}\label{definite} 
Let $F$ be a totally real number field, $\BA$ the ring of adeles and $\pi$ a cuspidal automorphic
representation of $\GL_{2}(\BA)$ with parallel weight two.
% with conductor $\frak{n}$.
Let $\fc \subset \cO_F$ be an ideal, $S=\supp(\fN\fc \infty)$ and $\chi_{0,S}$ a local character as above.
% with type $S$.
Let $\Theta_{S}$ be the set of CM quadratic extensions $K/F$ and $\fX_{K,\chi_{0,S}}$ the set of Hecke characters over $K$ with $S$-type $\chi_{0,S}$ as above.
%Let $S$ be a finite set of places of $F$ containing all archimedean places and the finite places dividing $\frak{n}$.
%Let $S_1\subset S$ be the set of places $v|\frak{n}\infty$ and $K_{0, v}$ non-split.
%For $U_{0,S}$ as above, let $\chi_{0,S}$ be a character of $U_{0,S}$.
%Let $i\in \{0,1\}$ be such that
%$$\prod_{v\in S_1} \epsilon(1/2, \pi_v, \chi_{0, v})\chi_{0, v}\eta_v(-1)= (-1)^{i}.$$
%$$
%\epsilon(\pi,\chi_{0})=(-1)^{i}
%$$
%for some $\chi_{0} \in \fX_{K,\kappa,\fc}$.
%For $\Theta \subset \Theta_{S}$ infinite, 

Then, we have

$$
%\lim_{K\in \Theta} 
\#\left\{\chi\in \fX_{K,\chi_{0,S}} \ \Big|\ L(1/2,\pi,\chi)\neq 0 \right\} \gg_{\pi, F, \epsilon, \chi_{0, S}} |D_K|^{\delta - \epsilon}%=\infty.
$$
for an absolute constant $\delta > 0$ and any $\epsilon > 0$.
\end{thm}
The approach is based on 
\begin{itemize}
\item[--] subconvex bound for the underlying Rankin--Selberg L-values,
\item[--] equidistribution of special points on the definite Shimura set and
\item[--] explicit version of the Waldspurger formula. 
\end{itemize}
\begin{proof} 
In view of the subconvex bound for the underlying Rankin--Selberg L-values due to Michel--Venkatesh (\cite[Thm. 1.2]{MV}), there exists an absolute constant $\delta > 0$ such that
\begin{equation}\label{scv}
L(1/2,\pi,\chi)\ll_{\pi,F} |D_{K}|^{1/2-\delta}.
\end{equation}

Based on an explicit Waldspurger formula (\cite[Thm. 1.9]{CST}), there exists a constant $C$ only dependent on our initial datum such that
\begin{equation}\label{exWd}
L(1/2, \pi, \chi)= C \cdot |D_K|^{-1/2}|P_f(\chi)|^2
\end{equation}
for any Hecke character $\chi\in \fX_{K,\chi_{0,S}}$. 

In view of an equidistribution result due to Michel (\cite{M}), we have
\begin{equation}\label{eqd}
\frac{1}{h_{K,\fc}}\sum_{\sigma\in C_{K}} |f(x_{\sigma})|^2= h_{K,\fc} (1+o(1)).
\end{equation}
Strictly speaking, the equidistribution in \cite[\S6]{M} concerns the case when $F=\BQ$ and $\fc=1$. The result can be generalised to our situation via the same method based on the subconvex bound \eqref{scv} and certain explicit Waldspurger formulae (\cite{CST}, for example \eqref{exWd}).

Let $N_{K,\chi_{0,S}}$ denote the number of $\chi\in \fX_{K,\chi_{0,S}}$ such that 
$L(1/2, \pi, \chi)\neq 0$. 

For $\epsilon >0$, it now follows that there exist constants $C'$ such that
$$\begin{aligned}
N_{K,\chi_{0,S}} \cdot |D_K|^{1/2-\delta} &\gg_{\pi,F} \sum_{\chi \in \fX_{K,\chi_{0,S}}} L(1/2, \pi, \chi)=
C \frac{h_{K,\fc}}{\sqrt{|D_K|}} \left(\frac{1}{h_{K,\fc}} \sum_{\sigma \in C_{K}}
|f(x_{\sigma})|^2\right)\\
&=C' \frac{h_{K,\fc}^2}{\sqrt{|D_K|}}(1+o(1))\gg_{\epsilon} |D_K|^{1/2-\epsilon}.\end{aligned}$$
Here we utilise \eqref{scv}, \eqref{exWd}, \eqref{eqd} and the Brauer--Siegel lower bound, respectively. 

This finishes the proof.

%Recall that $h_K\gg |D_K|^{1/2-\epsilon}$ for any $\epsilon>0$.

%for any Hecke character $\chi\in \fX_{K,\chi_{0,S}}$. 

\end{proof} 

\begin{remark}
(1). It is perhaps instructive to compare and contrast the indefinite and definite case. 
The approach seems quite specific to the case in hand and in particular we do not know if it can be transferrable. 

(2). In the case when $\pi$ does not have parallel weight two, the above approach does not seem to work directly as the test vector is no longer a function on the definite Shimura set. Nevertheless, the case may still be within reach.
\end{remark}

%We now consider the case that $B$ is totally definite quaternion algebra over $F$. We are going to show the Theorem \ref{definite} via analytic approach. 

%We thus obtain the following.
%\begin{thm}
%Let the notation and assumptions be as in Theorem \ref{definite}. Assume that for all infinite places $v$, $\epsilon(\pi_v, \chi_{0, v})\chi_v\eta_v(-1)=-1$.
%For $\epsilon > 0$ and $\delta$ as in \cite[Thm. 1.2]{MV},
%we have
%$$ \#\left\{\chi\in \fX_K \ \Big|\ L(1/2, \pi,  \chi)\neq 0 \right\}\gg_{\pi, F, \epsilon, \chi_{0, S}} |D_K|^{\delta-\epsilon}.$$
%\end{thm}

\section{Main results}
In this section, we conclude with the main results. 

\subsection{Torus embedding}
In this subsection, we describe a result on the existence of a certain torus embedding into a quaternion algebra. The embedding is used in the proof of the main results in \S5.2.

Let us first introduce the setup.

Let $F$ be a totally real number field and $\BA$ its ring of \adeles. Let $\pi$ be a cuspidal automorphic representation of $\GL_2(\BA)$ with conductor $\fN$ and  a finite order central character $\omega$. Suppose that $\pi_\infty$
is a discrete series for $\GL_2(F_\infty)$ of weight $(k_v)_{v|\infty}$ with the same parity. 
%Let $k_0=\max_v k_v$.
%Let $\phi\in S_2(\Gamma_0(N), \omega)$ be a newform of weight $2$, level $N$ and  Neben %character $\omega$. Let $\omega$ also denote its associated Hecke character on $\BA^\times$.
Let $\fc$ be an ideal of $\cO_{F}$ such that for $\fp |\fN$, we have 
$$\ord_\fp \omega_\fp\leq \ord_\fp \fc.$$ 
Here $\omega_\fp$ denotes the local component of $\omega$.

Let $B$ be quaternion algebra over $F$ 
%which is not totally definite 
such that there exists an irreducible automorphic  representation $\pi^{B}$ on $B_{\BA}^\times$ whose Jacquet--Langlands transfer is the automorphic representation $\pi$ of $\GL_2(\BA)$.

Let $S=\Supp (\fN\fc \infty)$.  Let $K_{0, S}\subset B_S$ be a $F_S$-subalgebra such that 
\begin{itemize}
\item[(i)] $K_{0, \infty}=\BC$ and 
\item[(ii)] $K_{0, v}/F_v$ is semi-simple quadratic.
\end{itemize}
Fix a maximal order $R^{(S)}$ of $B_{\BA}^{(S)}\cong M_2(\BA^{(S)})$.
Let $U^{(S)}=R^{(S)\times}$. Note that $U^{(S)}$ is a maximal compact subgroup of
$B_{\BA}^{\times (S)}\cong \GL_2(\BA^{(S)})$.

The main result of this subsection is the following 
\begin{lem}
\label{embedding}
Let $K/F$ be a CM quadratic extensions with $K_S\cong K_{0, S}$.

Then, there exists an embedding $\iota: K \hookrightarrow B$ such that
\begin{itemize}
\item[(i)] $\iota(K_{S}) = K_{0,S} $ and
\item[(ii)] $\iota(\BA_{K}^{S}) \cap U=\iota(\cO_{K}^{S})$.
\end{itemize}
\end{lem}
\begin{proof}
Let $\iota_{0}:K \hookrightarrow B$ be an embedding such that
$$
\iota(\BA_{K}^{S}) \cap U = \iota(\widehat{\cO}_{K}^{S}).
$$
For $v \in S$, we have $\iota(K_{v}) \simeq K_{v}$. Thus, there exists $b_{v}\in B_{v}^{\times}$
such that
$$
b_{v}^{-1}\iota(K_{v})b_{v}=K_{0,v}.
$$
Let $b_{S}=(b_{v})_{v \in S}\in B_{S} \subset B_{\BA}^{\times}$.

In view of the choice of the level $U$, strong approximation for the subgroup $B^{1} \subset B^\times$ consisting of elements with identity norm implies that 
$$
B_{\BA}^{\times}=B^{\times}\cdot U \prod_{v \in S} K_{0,v} .
$$
Thus, There exists $b \in \widehat{B}^{\times}$ such that
$$b_{S}=bu$$ with $b \in B^{\times}$ and $u \in U \prod_{v \in S} K_{0,v} $.

We finally take $\iota$ to be $\iota_{0}^{b}:=b^{-1}\iota_{0}b$.
\end{proof}
\begin{remark}
An analogue of the embedding holds in the case of an incoherent quaternion algebra over $\BA$. The argument goes through verbatim. 
\end{remark}
\begin{defn}\label{Theta'}
Let $\Theta_{S}$ denote the set of CM quadratic extensions $K$ such that
\begin{itemize}
\item[] $K_{S}\cong K_{0,S}$. 
%\item[(ii)] $p\nmid |\Pic_{K, c}|$.
%\item[(ii)] there exists an embedding $\iota_{K}:K\ra B$ with $K_S=K_{0, S}$ and
%\item[(iii)] $\wh{K}^{(S)}\cap R^{(S)}=\wh{\CO}_K^{(S)}$ under the embedding.
\end{itemize}
\end{defn}
For a CM quadratic extension $K/F$ with $K_S\cong K_{0, S}$, we fix an embedding $\iota_K$ as above from now.

\subsection{Main results} In this section, we describe and prove the main results. The proof is essentially a compilation of the results in \S2--\S4 and \S5.1.

Let the notation and assumptions be as in the introduction. In particular, $F$ is a totally real field and $\BA$ the ring of \adeles over $F$. 

We first consider the case of root number $-1$. Accordingly, we restrict to the case $\pi$ being a cuspidal automorphic representation of $\GL_2(\BA)$ with parallel weight two associated to a $\GL_2$-type abelian variety $A$ over the totally real field $F$.
% with $\cO_{L_{\phi}} \subset \End(A)$.

%Recall the embeddings $\iota_{\infty}:\overline{\BQ}\hookrightarrow \BC$. 
Our result regarding the non-vanishing is the following 
\begin{thm}\label{main'1}
Let $F$ be a totally real number field, $\BA$ the ring of adeles and $\pi$ a cuspidal automorphic
representation of $\GL_{2}(\BA)$ with parallel weight two.
Suppose that $\pi$ corresponds to a $\GL_2$-type abelian variety over $F$.
For an integral ideal $\fc$, let $\Theta_{\fc}$ be  an infinite set of CM quadratic extensions $K/F$ such that there exists a finite order Hecke character
$\chi_0$ over $K$ satisfying the conditions (C1) and (C2) along with
$$
\epsilon(\pi,\chi_0)=-1.
$$

For $\Theta \subset \Theta_\fc$ infinite, we have
$$\lim_{K\in \Theta} \#\Big\{ \chi\in \fX_{K,\fc} \Big|\ L'(1/2, \pi, \chi)\neq 0 \Big\}=\infty.$$
%\begin{cor}
\end{thm}
\begin{proof} 
%In view of \cite{YZZ}, we have finitely many choices for incoherent quaternion algebras $\BB$ over $\BA$ such that $\pi$ transfers to $\BB^{\times}$. We now fix such an incoherent quaternion algebra $\BB$. 

Let $S=\supp(\fN\fc\infty)$ as before. 

Let $K_{0, S}$ be a $F_S$-subalgebra such that $K_{0, \infty}=\BC$ and $K_{0, v}/F_v$ is semi-simple quadratic. 
Let $U_{0,S}$ be as in \S3.1. 
%$$ U_{0, S}:=\prod_{v\in S,\ \text{$v$  split}}\CO_{K_{0, v}}^\times \times \prod_{v\in S,\ \text{$v$ non-split}} K_{0, v}^\times.$$
Suppose that there exists a finite order character $\chi_{0, S}: U_{0, S}\lra \ov{\BQ}^\times$ with conductor $\fc$ satisfying (LC1)-(LC2) in \S3.1 such that 
$$
\epsilon(\pi, \chi_{0,v})\chi_{0,v}\eta_{0,v}(-1)=-1
$$
for odd number of $v \in S$. 

From now, let $K \in \Theta_{S}$ (Definition \ref{Theta'}).
%For any $v\in S$, recall that we say $v$ is non-split if $K_{0, v}$ is a field and split otherwise. 
Recall that $\fX_{K,\fc}$ denotes the set of finite order Hecke characters over $K$ satisfying the conditions (C1) and (C2) (\S1).
We have 
$$
\fX_{K,\fc}= \fX_{K,\fc}^{-} \bigcup \fX_{K,\fc}^{+}.
$$
Here $\fX_{K,\fc}^{-}$ (resp. $\fX_{K,\fc}^{+}$) denotes the subset of Hecke characters $\chi$ such that $\epsilon(\pi,\chi)=-1$ (resp. $\epsilon(\pi,\chi)=1$). In view of the hypothesis, $\fX_{K,\fc}^{-}$ is non-empty.

%Let $S=\supp(\fN\fc\infty)$.
%For any $v\in S$, recall that we say $v$ is non-split if $K_{0, v}$ is a field and split otherwise. 
%Let $U_{0,S}$ be as in \S3.1.
%$$ U_{0, S}:=\prod_{v\in S,\ \text{$v$  split}}\CO_{K_{0, v}}^\times \times \prod_{v\in S,\ \text{$v$ non-split}} K_{0, v}^\times.$$
%Suppose we are given a finite order character $\chi_{0, S}: U_{0, S}\lra \ov{\BQ}^\times$ with conductor $\fc$ satisfying (LC1)-(LC3) in \S3.1. 
As the conductor of Hecke characters in $\fX_{K,\fc}$ is fixed,
we have
$$
\fX_{K,\fc}^{-} = \bigcup_{\chi_{0,S}} \fX_{K,\chi_{0,S}}.
$$
Moreover, there are only finitely many choices for $(S,\chi_{0,S})$.
Here $\fX_{K,\chi_{0,S}}$ denotes the set of Hecke characters as in Definition \ref{char}. 

For any such $\chi_{0,S}$, we consider an incoherent quaternion algebra $\BB$ with $\ram(\BB) \subset S$ such that 
$$
\epsilon(\pi, \chi_{0,v})\chi_{0,v}\eta_{0,v}(-1)=\epsilon(\BB_{v})
$$
for $v \in S$. 

%In the case of root number $1$, 
The result now follows from Lemma \ref{embedding} and Theorem \ref{derivative}.
\end{proof}

We now consider the case of root number $+1$. Accordingly, we consider the general case $\pi$ being a cuspidal cohomological automorphic representation of $\GL_2(\BA)$. 

Our result regarding the non-vanishing is the following 

\begin{thm}\label{main'0}
Let $F$ be a totally real number field, $\BA$ the ring of adeles and $\pi$ a cuspidal automorphic
representation of $\GL_{2}(\BA)$.
Suppose that $\pi_\infty$
is a discrete series for $\GL_2(F_\infty)$ of weight $(k_v)_{v|\infty}$ with the same parity.
%Suppose that $\pi$ corresponds to a $\GL_2$-type abelian variety over $F$.
For an integral ideal $\fc$, let $\Theta_{\kappa,\fc}$ be  an infinite set of CM quadratic extensions $K/F$ such that there exists a Hecke character
$\chi_{0} \in \fX_{K,\kappa,\fc}$ over $K$ satisfying the conditions (C1) and (C2) along with
$$
\epsilon(\pi,\chi_0)=+1.
$$
If $k_{\sigma} > l+2\kappa_{\sigma}$ for all $\sigma \in I$, then suppose that $\pi$ is with parallel weight $2$.

For $\Theta \subset \Theta_{\kappa,\fc}$ infinite, we have
$$\lim_{K\in \Theta} \#\Big\{ \chi\in \fX_{K,\kappa,\fc} \Big|\ L(1/2, \pi, \chi)\neq 0 \Big\}=\infty.$$
\end{thm} 
\begin{proof} 
Based on Lemma \ref{embedding} and Theorem \ref{derivative}, the same argument as in the proof of Theorem \ref{main'1} applies.
\end{proof}

\end{document}